\documentclass[12pt]{article}
\usepackage{amsmath,amsfonts,amssymb,amsthm}
\usepackage{graphics}
\usepackage{epsfig}
\usepackage{xcolor}

\usepackage[nottoc]{tocbibind}

\newcommand\numberthis{\addtocounter{equation}{1}\tag{\theequation}}

\newcommand{\ignore}[1]{}

\newtheorem{proposition}{Proposition}
\newtheorem{theorem}{Theorem}
\newtheorem{remark}{Remark}
\newtheorem{lemma}{Lemma}

\newcommand{\R}{\mathbb{R}}
\newcommand{\Z}{\mathbb{Z}}

\newcommand{\F}{\mathcal{F'}}

\newcommand{\ES}{\mathcal{S}}
\def\Ra{{\rm Ra}}
\def\Pr{{\rm Pr}}
\def\Nu{{\rm Nu}}
\def\Re{{\rm Re}}
\xdefinecolor{darkgreen}{rgb}{0,0.55,0}

\author{Antoine Choffrut, Camilla Nobili and Felix Otto}
\title{Upper bounds on Nusselt number at finite Prandtl number}

\begin{document}

\maketitle

\begin{abstract}
We study Rayleigh B\'enard convection based on the Boussinesq approximation.
We are interested in upper bounds on the Nusselt number $\Nu$, the upwards
heat transport, in terms of the Rayleigh number $\Ra$, that characterizes 
the relative strength of the driving mechanism and the Prandtl number 
$\Pr$, that characterizes the strength of the inertial effects.
We show that, up to logarithmic corrections, 
the upper bound $\Nu\lesssim \Ra^{\frac{1}{3}}$ from \cite{CD99}
persists as long as $\Pr\gtrsim\Ra^{\frac{1}{3}}$ and then 
crosses over to $\Nu\lesssim\Pr^{-\frac{1}{2}}\Ra^{\frac{1}{2}}$. 
This result improves the one of Wang \cite{XW} by going beyond the 
perturbative regime $\Pr\gg \Ra$. The proof uses a new way to estimate 
the transport nonlinearity in the Navier Stokes equations capitalizing 
on the no-slip boundary condition. It relies on a new Calder\'on-Zygmund 
estimate for the non-stationary Stokes equations in $L^1$ with a borderline 
Muckenhoupt weight.

\smallskip


\end{abstract}

\tableofcontents

\newpage

\section{Introduction}
\subsection{Background}

Rayleigh-B\'enard convection is the buoyancy-driven flow of a fluid heated from below and cooled from above.
This model of thermal convection, besides having some important application in 
geophysics, astrophysics, meteorology, oceanography and engineering, is a paradigm for nonlinear and chaotic dynamics, pattern formation and 
turbulence.
\newline
A fluid is enclosed between two rigid parallel plates separated by a vertical  distance $h$ and held
at different temperatures $T=T_{\rm{bottom}}$ and $T=T_{\rm{top}}$ at height $0$ and $h$ respectively, with $T_{\rm{bottom}}>T_{\rm{top}}$.
\newline
In a $d-$dimensional container we follow the evolution equations of the  velocity vector field $u(x,t)$, 
the  temperature scalar field $T(x,t)$ and the  pressure scalar field $p(x,t)$ where we indicate with $x$ the $d-$dimensional spatial
variable and with $t$ the time variable. In what follows we specify with $x'$ the first $d-1$ horizontal 
components and with $z$ the vertical component of the vector $x$.
In the Oberbeck-Boussinesq approximation,
where variations of the density of the fluid $\rho$ are ignored except in the buoyancy term, $u,T$ and $p$
are governed by 
\begin{equation}\label{NS1}
  \left\{\begin{array}{rclc}
  \partial_t T+u\cdot \nabla T -\chi\Delta T&=& 0  \qquad & {\rm for } \quad 0<z<h\,,\\
   \partial_t u+(u\cdot \nabla) u-\nu\Delta u+\nabla p&=&g\alpha T e_z \qquad & {\rm for } \quad 0<z<h \,,\\
   \nabla\cdot u&=& 0 \qquad & {\rm for } \quad 0<z<h\,,\\
    u&=&0  \qquad & {\rm for } \quad z\in\{0,h\}\,,\\
    u&=&u_0 \qquad & {\rm for } \quad t=0 \,,\\
    T&=&T_{bottom} \qquad & {\rm for } \quad z=0\, ,\\
    T&=&T_{top} \qquad & {\rm for } \quad z=h\, ,\\
   \end{array}\right.
 \end{equation}
 where $e_z$ is the upward unit normal vector and the parameters appearing are
   the kinematic viscosity $\nu$, the acceleration of gravity $g$,
 the thermal expansion coefficient $\alpha$ and  
 the thermometric conductivity $\chi$. We notice that $\chi=\frac{\kappa}{\rho_0 c_p}$
 where $\kappa$ is the thermal conductivity, $\rho_0$ is the constant density and $c_p$ is
 the specific heat at constant pressure ( see \cite{Landau-Lifshitz}, chapter $5$).
\newline
The temperature, which is set to be higher at the bottom
plate than at the top plate, diffuses $(-\Delta T)$ and is
  advected by the velocity $(u\cdot\nabla T)$. 
  The fluid, which to leading order is incompressible ($\nabla\cdot u=0$)
  expands in response to a change in temperature (the origin of the term $\alpha T$),
  which in presence of gravity leads to the buoyancy term $g\alpha Te_z$.
  This accelerates the fluid $(\partial_t u+u\cdot \nabla u$ in Eulerian coordinates).
  The acceleration is eventually balanced by viscosity ($-\nu\Delta u$) in 
  conjunction with the no-slip boundary condition ($u=0$ at $z=0$ and $z=h$).
 The pressure $p$ appears as a Lagrangian multiplier to enforce the 
 divergence-free condition. 
Periodicity in the horizontal variables $x'\in[0,L)^{d-1}$ is imposed for all the functions.
\newline
If the driving forces, as measured by the Rayleigh number defined below, are small, the pure
conduction solution ($T=T_{\rm bottom}-(T_{\rm bottom}-T_{\rm top})\frac{z}{h}$, $u=0$) is the 
global attractor. Above an explicitly known critical Rayleigh number, the conduction solution 
is unstable and the global attractor consists of stationary {\it convection rolls} (see \cite{Ahlers} for historical context).
As the Rayleigh number increases further, the stationary convection rolls become unstable.
For sufficiently high Rayleigh number, the temperature field features boundary layer, from which plumes detach.
This is what, in this context, is called the turbulent regime.

The quantity of interest is the averaged upward heat flux. An inspection
of the advection-diffusion equation for the temperature shows that the
heat flux is given by $\rho_0c_p(Tu-\chi\nabla T)$
\footnote{Writing the equation of the temperature in divergence form 
we isolate the term $Tu-\nabla T$, which turns out
to have the dimensions of heat flux after multiplication with $\rho_0c_p$.}. The appropriately non-dimensionalized
measure of the time-space average of the upward heat flux is given 
by the Nusselt number $\Nu$ defined through
 {\small
 \begin{equation}\label{Nusselt1}
 \Nu=\frac{h}{\rho_0c_p\chi(T_{bottom}-T_{top})}\lim_{t_0\uparrow\infty}\frac{1}{t_0}\int_0^{t_0}\frac{1}{h}\int_0^h\frac{1}{L^{d-1}}\int_{[0,L)^{d-1}}\rho_0c_p(Tu-\chi\nabla T)\cdot e_zdx'dzdt.
 \end{equation}}
Note that the Nusselt number is normalized by the term 
$\rho_0c_p\frac{\chi(T_{bottom}-T_{top})}{h},$ which is the vertical heat flux associated to 
the pure conductive state  $u=0$ and $T=T_{bottom}-(T_{bottom}-T_{top})\frac{z}{h}$.
In other words, the pure conductive solution gives rise to $\Nu=1$.
\newline
For generic initial data $\Nu$ is thought to satisfy 
a "similarity law", that is, to be a function of 
the aspect ratio $\frac{L}{h}$ of the container (provided the artificial 
period $L$ is assimilated to the lateral width of the container) and the two 
non-dimensional parameters $\Ra$ and $\Pr$ :
The Rayleigh number $\Ra$ is defined as 
$$\Ra=\frac{g\alpha(T_{\rm{bottom}}-T_{\rm{top}})h^3}{\nu\chi}$$
and the Prandtl number $\Pr$ is the ratio of two diffusivities
$$\Pr=\frac{\nu}{\chi}.$$
While the Rayleigh number is a system parameter which 
expresses the relative strength of the driving mechanisms
(temperature differences, thermal expansion and gravity), the Prandtl number 
depends only on the fluid (and its absolute temperature), and is thus more arduous to vary in 
the experiments. It can vary from very small number
(e.g. $0.015$  for mercury) to large and very large numbers (e.g. $13.4$ for seawater and $10^{24}$ for Earth's mantle).
\newline
In the non-dimensionalization (\ref{NSE}), the aspect ratio can be assimilated
to the artificial lateral periodicity $L$. This paper does not address
the dependency on the aspect ratio: We will focus on (upper) bounds that
are independent of the period $L$, which amounts to neglecting the 
(limiting) effects of the lateral boundary conditions. 
Likewise, this paper does not address the specifics of two-dimensional flows: 
Our analysis is in fact oblivious to the dimension $d$, 
which in particular amounts to allowing for turbulent boundary layers.

\medskip

There are classical heuristic arguments in favor of two (different) scaling laws for
$\Nu$ in terms of $\Ra$ and $\Pr$: 

\medskip
\noindent
Spiegel in \cite{Sp} realized that in the regime in which $\Nu$ is independent of the quantities
characterizing both dissipative mechanisms, namely the kinematic viscosity $\nu$ 
and the thermometric conductivity $\chi$, the only possible scaling is
$\Nu\sim(\Pr\,\Ra)^\frac{1}{2}$ . Indeed, by the above definitions of $\Ra$, $\Pr$, and $\Nu$, 
this scaling translates into
\begin{equation}\nonumber
\lim_{t_0\uparrow\infty}\frac{1}{t_0}\int_0^{t_0}\frac{1}{h}\int_0^h\frac{1}{L^{d-1}}
\int_{[0,L)^{d-1}}{\textstyle\frac{T-T_{top}}{T_{bottom}-T_{top}}}u\cdot e_zdx'dzdt
\sim \left(g\alpha(T_{bottom}-T_{top})h\right)^\frac{1}{2},
\end{equation}
where we neglected the conductive contribution to the heat flux
and used incompressibility (together with the no-flux boundary conditions) 
in form of $\int_{[0,L)^{d-1}}u\cdot e_zdx'=0$. Now the left-hand side
is an average upward velocity of the warm fluid parcels whereas the right-hand side assumes the form of
$(\mbox{acceleration}\times\mbox{height})^\frac{1}{2}$, where the acceleration
$g\alpha(T_{bottom}-T_{top})$ is the (relative) upward acceleration
of warm fluid parcels due to their density reduction of $\alpha(T_{bottom}-T_{top})$ in the presence
of gravity $g$. In this sense, this scaling is the scaling of free fall, or rather, free rise (in (\ref{rescaling2})
we give a more mathematical heuristic argument for this scaling regime starting from the 
non-dimensionalized equations ).

\medskip
\noindent
When the inertia of the fluid is neglected, i.\ e.\ $\Pr=\infty$,
Malkus \cite{Ma} proposed the following heuristic argument in favor of the scaling 
 $\Nu\sim \Ra^\frac{1}{3}$ : 
It starts with the observation that for $\Ra\gg 1$, in a boundary
layer of (to be determined) thickness $\delta\ll h$, the temperature drops from $T_{bottom}$ to its
average $\frac{1}{2}(T_{bottom}+T_{top})$ and the flow is suppressed. By definition of $\Nu$,
this yields
\begin{equation}\nonumber
\Nu\approx \frac{h}{2\delta},
\end{equation}
so that $\Nu$ is linked to the relative size of the (thermal) boundary layer.
Incidentally, an upper bound motivated by this relation is the starting point in the rigorous treatment,
cf (\ref{N3}).
Here comes the crucial argument of a {\it marginally stable} boundary layer: The actual
size $\delta$ is expected to be proportional to the largest height $h^*$ of container in which
the pure conductive solution $T=T_{bottom}+(T_{top}-T_{bottom})\frac{z}{h^*}$ and $u=0$ is stable.
This critical height $h^*$ (critical in both the sense of linear and nonlinear stability) is 
explicitly known. Here it suffices to appeal to a dimensional argument which yields
that the critical Rayleigh number $\Ra^*$ must be universal, so that $\Ra^*\sim 1$, 
which in view of its definition translates into
$h^*\sim(\frac{g\alpha(T_{bottom}-T_{top})}{\nu\chi})^{-\frac{1}{3}}$ so that we obtain
\begin{equation}\nonumber
\delta\sim \Ra^{-\frac{1}{3}}h.
\end{equation}
The combination yields the desired $\Nu\sim \Ra^\frac{1}{3}$. We note that this implies
{\small\begin{equation}\nonumber
\lim_{t_0\uparrow\infty}\frac{1}{t_0}\int_0^{t_0}\frac{1}{h}\int_0^h\frac{1}{L^{d-1}}
\int_{[0,L)^{d-1}}(Tu-\chi\nabla T)\cdot e_zdx'dzdt
\sim \left(\frac{g\alpha\chi^2(T_{bottom}-T_{top})^4}{\nu}\right)^\frac{1}{3},
\end{equation}}
so that in this regime the heat transport is independent of the container height $h$ (see (\ref{rescaling1})
for a more mathematical heuristic argument).

\medskip
\noindent
Many more scaling regimes for $\Nu$ in the $\Pr$-$\Ra$-plane have been proposed
 on experimental and theoretical grounds
in the physics literature.
By means of  mixing length theory, Kraichnan in \cite{Kra} not only 
reproduced the scalings of Malkus and Spiegel for big and very small $\Pr$ respectively,
but also suggested a third scaling
$\Nu\sim \Pr^{-\frac{1}{4}}\Ra^{\frac{1}{2}}$ for big $\Ra$ and moderately low $\Pr$.
A fairly complete theory has been worked out in \cite{Gro-Lo}. 
It is based on global balance laws (which we also
use in our rigorous treatment (see Section \ref{sec2}) 
on distinguishing the cases of the dissipation dominantly taking place in the bulk or in the boundary layer) 
and on assumptions on the structure of both the thermal and the viscous boundary layer (which becomes relevant
for $\Pr<\infty$). However,
these statements are more speculative when the viscous boundary layer is turbulent rather than laminar.

\medskip

Despite the complexity of the phenomenon of Rayleigh-B\'enard convection in the
turbulent regime, there are rigorous upper bounds of $\Nu$ in terms of $\Ra$ and $\Pr$.
In their 1996 paper \cite{CD96}, Constantin \& Doering among other results
gave an easy argument for $\Nu\lesssim \Ra^\frac{1}{2}$ for all values of $\Pr$. 
On the one hand, this bound is suboptimal for $\Pr\gg 1$ (as our result implies), which is
not surprising since the argument is oblivious to replacing the no-slip boundary condition
by a no-stress boundary condition. On the other hand, it does not capture the 
Spiegel scaling of $\Nu\sim (\Pr \Ra)^\frac{1}{2}$ in the inviscid limit $\Pr\ll 1$.
In their seminal 1999 paper \cite{CD99}, Constantin \& Doering considered the case
of $\Pr=\infty$ and proved $\Nu\lesssim (\Ra\ln^2\Ra)^\frac{1}{3}$. They obtained this bound by
combining global balance laws with the maximum principle for the temperature and a (logarithmically failing)
maximal regularity estimate for the (quasi)-stationary Stokes equations in $L_x^\infty$. The present paper
follows and extends the strategy laid out in this work. In 2006, Charles Doering, Maria Westdickenberg
(n\'ee Reznikoff)  and the last author \cite{DOR} obtained the same bound (with a slightly improved
power of the logarithm) with help of a strategy that is inspired by the above heuristic argument
of marginal stability of the boundary layer, namely the background field method. However, this
method is not capable of giving the optimal bound: On the one hand, this saddle point 
method cannot give a better bound than $(\Ra\ln^\frac{1}{15}\Ra)^\frac{1}{3}$ \cite{NO};
on the other hand, the combination of arguments in \cite{CD99} and \cite{DOR} yields the 
doubly logarithmic bound $\Nu\lesssim (\Ra\ln\ln \Ra)^\frac{1}{3}$, cf.\ \cite{Otto-Seis}.

\smallskip

In the case of $\Pr<\infty$, the lack
of  instantaneous {\it slaving} of the velocity 
 field to the temperature field increases the difficulty 
 in  bounding the convection term $\int\langle u^zT\rangle dz$ 
 in the definition of the Nusselt number and
the background field  method turns out to be no longer fruitful.
Besides \cite{CD96}, there is only one other rigorous result for $\Pr<\infty$: Wang \cite{XW}
proved by a perturbative argument that the Constantin \& Doering 1999 bound $\Nu\lesssim (\Ra\ln^2\Ra)^\frac{1}{3}$ 
persists for $\Pr\gg \Ra$ (see (\ref{perturb}) for an argument why this is the classical scaling regime). 
Like Wang's argument, ours treats the convective nonlinearity $(u\cdot\nabla) u$ perturbatively.
However, there is a difference: We perturb around the {\it non-stationary} Stokes equations
and gain access to $\Ra$-$\Pr$-regimes where the effective {\it Reynolds} number $\Re$ is allowed to be large.
In fact
we work with Leray's solution and thus only appeal to the global energy estimate on the level of the
Navier-Stokes equations, whereas Wang's regime is limited by the use of the small-data regularity theory
for the Navier-Stokes equations and thus $\Re\ll 1$, which in his analysis translates into $\Pr\gg \Ra$ . 
Our more robust strategy allows us to show that the Constantin \& Doering 1999 bound $\Nu\lesssim (\Ra\ln^2\Ra)^\frac{1}{3}$
(in its slightly improved form of $\Nu\lesssim (\Ra\ln \Ra)^\frac{1}{3}$) persists in the much larger
regime $\Pr\gtrsim (\Ra \ln \Ra)^\frac{1}{3}$ and then crosses over to $\Nu\lesssim (\frac{\Ra\ln \Ra}{\Pr})^\frac{1}{2}$,
which can be seen as an interpolation between the marginal stability bound and the Constantin \& Doering 1996
bound as $\Pr$ decreases from large $\Pr=(\Ra \ln \Ra)^\frac{1}{3}$ to moderate $\Pr=1$.
Loosely speaking our analysis just requires small $\Re$ in the thermal boundary layer, not in the entire 
container, for the $(\Ra\ln \Ra)^{\frac{1}{3}}$ scaling to persist.

\subsection{Main Results}
Measuring lengths in units of the layer depth $h$, time in units of the thermal diffusion $h^2/\chi$,
 and temperature on a scale where $T_{\rm{top}}=0$ and $T_{\rm{bottom}}=1$, we 
 non-dimensionalize the problem (\ref{NS1}) and consider
\begin{equation}\label{NSE}
   \left\{\begin{array}{rclc}
   \partial_t T+u\cdot \nabla T -\Delta T&=& 0  \qquad & {\rm for } \quad 0<z<1\,,\\
    \frac{1}{\Pr}(\partial_t u+(u\cdot \nabla) u)-\Delta u+\nabla p&=&\Ra T e_z \qquad & {\rm for } \quad 0<z<1  \,,\\
    \nabla\cdot u&=& 0 \qquad & {\rm for } \quad 0<z<1\,,\\
     u&=&0  \qquad & {\rm for } \quad z\in\{0,1\}\,,\\
     u&=&u_0 \qquad & {\rm for } \quad t=0 \,,\\
     T&=&1 \qquad & {\rm for } \quad z=0\,,\\
     T&=&0 \qquad & {\rm for } \quad z=1\,.\\
    \end{array}\right.
  \end{equation}
 %
 In terms of the dimensionless variables in equation (\ref{NSE}) the Nusselt
 number turns to be 
  {\small
 \begin{equation}\nonumber
 \Nu=\limsup_{t_0\uparrow\infty}\frac{1}{t_0}\int_0^{t_0}\int_0^1\frac{1}{L^{d-1}}\int_{[0,L)^{d-1}}(Tu-\nabla T)\cdot e_zdx'dzdt,
 \end{equation}}
 \newline
 where we consider the limit superior in order to avoid the 
 cases in which the limit does not exist.
\newline
 For the problem (\ref{NSE}) we establish the following result
\begin{theorem}[Bounds on the Nusselt number]\label{th2}\ \\
  Provided the initial data satisfy $T_0\in [0,1]$, $\int|u_0|dx<\infty$ and for $\Ra \gg 1$
  \begin{equation}\label{BOUND-NU}
    \Nu\leq C
       \begin{cases}
	  (\Ra\ln\Ra)^{\frac{1}{3}}\qquad & {\rm for } \quad  \Pr\geq (\Ra\ln \Ra)^{\frac{1}{3}},\\
	 \left(\Pr^{-1}\Ra\ln\Ra\right)^{\frac{1}{2}} \qquad & {\rm for } \quad  \Pr\leq (\Ra\ln \Ra)^{\frac{1}{3}},
       \end{cases}
  \end{equation}
  where $C$ depends only on the dimension $d$.
\end{theorem}
\noindent
In our analysis, the crucial no-slip boundary condition is both a blessing and a curse,
as we shall presently explain.
The no-slip boundary condition is a {\it blessing} because, via Hardy's inequality, it gives us good control of
the convective nonlinearity $(u\cdot\nabla) u$ near the boundary in terms of the average viscous dissipation 
$\int_{0}^{1}\langle|\nabla u|^2\rangle dz$ \footnote{See Section \ref{notations} for notations.},
which is the physically relevant quantity (and the only bound at hand for the Leray solution)
see (\ref{N2}). 
This yields control of $(u\cdot\nabla) u$ in an $L^1$-type space with the weight $\frac{1}{z(1-z)}$.
Hence a maximal regularity theory for the non-stationary Stokes equations with respect to this norm
is required. Since this norm is borderline for Calder\'on-Zygmund estimates (both because the exponent and
the weight are borderline), maximal regularity ``fails logarithmically'' and can only be recovered 
under bandedness assumptions (i.\ e.\ a restriction to a packet of wave numbers in Fourier space)
--- this is the source of the logarithm.
It is in this maximal regularity theory where the {\it curse} of the no-slip boundary condition
appears: As opposed to the no-stress boundary condition in the half space, the no-slip boundary condition
does not allow for an extension by reflection to the whole space, and thereby the use of simple kernels
or Fourier methods also in the normal variable. 
The difficulty coming from the the no-slip
boundary condition in the non-stationary Stokes equations when deriving maximal regularity estimates is of course 
well-known; many techniques have been developed to derive Calder\'on-Zygmund estimates despite this difficulty.
In the {\it{half space}} Solonnikov in \cite{solonnikov1964} has constructed a solution formula 
for (\ref{STOKES-STRIP}) with zero initial data via the Oseen an Green tensors. 
 An easier and more compact representation of the solution to the problem (\ref{STOKES-STRIP}) with zero forcing term and non-zero initial value was 
 later given by Ukai in \cite{Ukai} by using a different method. Indeed he could  write an explicit formula
 of the solution operator as a composition of Riesz' operators and solutions operator for the heat and Laplace's
equation. This formula is an effective tool to get $L^p-L^q$ ($1<q,p<\infty$) estimates for the solution and its derivatives.
 In the case of {\it exterior domains}, Maremonti and Solonnikov \cite{Mar-Sol} derive
  $L^p-L^q$ ($1<q,p<\infty$) estimates for (\ref{STOKES-STRIP}), going through estimates 
 for the extended solution in the half space and in the whole space.
   In particular in the half space they propose a decomposition of (\ref{STOKES-STRIP})
 with non-zero divergence equation.
The book of Galdi \cite{galdi} provides with a complete treatment of the classical theory and
 results on the non-stationary Stokes equations and Navier-Stokes equations.

The new element here is that we need maximal regularity in the {\it borderline} space
$L^1(dtdx'\frac{1}{z(1-z)}dz)$, and in $L^\infty_z(L^1_{t,x'})$ (the latter space coming
from the original argument in 1999 paper of Constantin \& Doering and pertaining to the buoyancy term). As mentioned,
these borderline Calder\'on-Zygmund estimates can only hold under
bandedness assumption. 

 \begin{theorem}[Maximal regularity in the strip]\label{th1}\ \\ 
      There exists $R_0\in(0,\infty)$ depending only on $d$ and $L$ such that the following holds.     
      Let $u,p,f$ satisfy           
      \begin{equation}\label{STOKES-STRIP}
	\left\{\begin{array}{rclc}
	 \partial_t u-\Delta u+\nabla p&=&f  \qquad & {\rm for } \quad 0<z<1 \,,\\
	  \nabla\cdot u&=& 0                 \qquad & {\rm for } \quad 0<z<1\,,\\
	   u &=& 0                           \qquad & {\rm for } \quad z\in\{0,1\}\,,\\
	   u &=& 0                           \qquad & {\rm for } \quad t=0\,.\\
	  \end{array}\right.        
         \end{equation}	
      Assume $f$ is horizontally  band-limited , i.e
      \begin{equation}\label{BANDLIM}
      \F f(k',z,t)=0 \mbox{ unless } 1\leq R|k'|\leq 4  \mbox{ where } R<R_0.
      \end{equation}
      Then,
      \begin{equation}\label{MRE}
      ||(\partial_t -\partial_z^2)u'||_{(0,1)}+||\nabla'\nabla u'||_{(0,1)}+||\partial_t u^z||_{(0,1)}+||\nabla^2 u^z||_{(0,1)}+||\nabla p||_{(0,1)}\lesssim||f||_{(0,1)},
      \end{equation}
      where $||\cdot||_{(0,1)}$ denotes the norm
      \begin{equation}\label{NORM-STRIP}
       ||f||_{(0,1)}:=||f||_{(R,(0,1))}=\inf_{f=f_0+f_1}\left(\sup_{0<z<1}\langle |f_0|\rangle+\int_0^{1}\langle |f_1|\rangle \frac{dz}{(1-z)z}\right)\,,
      \end{equation}
      where $f_0$ and $f_1$ satisfy the bandedness assumption (\ref{BANDLIM}).
  \end{theorem}
  \noindent
See Section \ref{notations} for notation.
\newline
Our analysis offers two insights:

\medskip
\noindent
It turns out that for the maximal regularity estimate for the quantity of interest, namely the second vertical derivative
$\partial_z^2u^z$ of the vertical velocity component $u^z=u\cdot e_z$, bandedness only in the {\it horizontal} variable $x'$ is required.
This is extremely convenient, since the horizontal Fourier transform (or rather, series), 
with help of which bandedness is expressed, is compatible with the lateral periodic boundary conditions.

\medskip
\noindent
It turns out that maximal regularity is naturally expressed in terms of the {\it interpolation}
between the two norms of interest $L^1(dtdx'\frac{1}{z(1-z)}dz)$ and $L^\infty_z(L^1_{t,x'})$.
\footnote{See Section \ref{notations} for notations.}
This way, one avoids the logarithm one would expect to be the price
of the borderline weight $\frac{1}{z(1-z)}$. It is a pleasant coincidence that the norm
$L^\infty_z(L^1_{t,x'})$ arises for two unrelated reasons: It is needed to estimate the buoyancy term $Te_z$ 
driving the Navier-Stokes equations and it is the natural partner of $L^1(dtdx'\frac{1}{z(1-z)}dz)$
in the maximal regularity estimate.


Aside from their application to this problem (see Section \ref{ProofTh1})
 all the estimates in Theorem \ref{th2} might have 
an independent interest since they
show the full extent of what one can 
obtain under the horizontal bandedness assumption only.
\medskip

{\bf Acknowledgments}: We thank Charles Doering for many stimulating discussions on Rayleigh-B\'enard
convection in general, and, more specifically, for bringing the $\Pr$-scaling to our attention.

\subsection{Idea of the proof}\label{sec2}
We first notice that for the maximum principle applied to the 
temperature equation we have 
\begin{equation}\label{Max-Princ}
 \mbox{if } T_0\in [0,1] \; \mbox{ then }\;  ||T||_{L^{\infty}}\leq 1\,,
\end{equation}
which furnishes us an a-priori bound 
on the temperature $T$.
\newline
One can verify that $\Nu$ as defined in (\ref{LTaHA}) 
satisfies 
\begin{equation}\label{heatflux-property}
 \Nu=\langle Tu^z-\partial_z T \rangle \;\; \forall z\in (0,1)\,
\end{equation}
(see \cite{Otto-Seis}) and thus in particular
\begin{equation}\label{Nusselt}
 \Nu=\int_0^1\langle(Tu-\nabla T)\cdot e_z\rangle dz.
\end{equation}

As a side remark, we now argue that the scaling laws predicted by Malkus and Spiegel, respectively, can be simply deduced 
by rescaling the equations in the limiting case when the viscosity term wins over the inertial term
 and vice versa.
 On one hand, if  we assume that the inertial term is negligible (setting $\Pr=\infty$)
the equation (\ref{NSE}) reduces to
\begin{equation*}
   \left\{\begin{array}{rclc}
   \partial_t T+u\cdot \nabla T -\Delta T&=& 0  \,,\\
    -\Delta u+\nabla p&=&\Ra T e_z \,,\\
    \nabla\cdot u&=& 0 \,.\\
    \end{array}\right.
  \end{equation*}
Rescaling this equation according to
\begin{equation}\label{rescaling1}
x=\Ra^{-\frac{1}{3}}\hat x, \; t=\Ra^{-\frac{2}{3}}\hat t,\; u=\Ra^{\frac{1}{3}}\hat u, \; p=\Ra^{\frac{2}{3}}\hat p \; \mbox{ and thus } \;\Nu=\Ra^{\frac{1}{3}}\widehat{\Nu}
\end{equation}
we end up with the parameters-free system  
\begin{equation*}
   \left\{\begin{array}{rclc}
   \partial_{\hat t}  T+\hat u\cdot \hat{\nabla}  T -\hat{\Delta}  T&=& 0  \,,\\
    -\Delta \hat u+\nabla \hat p&=&  T e_z \,,\\
    \hat{\nabla}\cdot \hat u&=& 0 \,,\\
    \end{array}\right.
  \end{equation*}
  which naturally lives in the half space.
  Since for the latter system, it is natural to expect that the heat flux is universal, i.e. $\widehat{\Nu}\approx 1$, we obtain
  $\Nu\sim \Ra^{\frac{1}{3}}.$
  \newline
On the other hand, if we rewrite the system (\ref{NSE}) neglecting the diffusivity and the viscosity term
\begin{equation*}
   \left\{\begin{array}{rclc}
   \partial_t T+u\cdot \nabla T &=& 0 \,,\\
    \frac{1}{\Pr}(\partial_t u+u\cdot \nabla u)+\nabla p&=&\Ra T e_z  \,,\\
    \nabla\cdot u&=& 0  \,,\\
    \end{array}\right.
  \end{equation*}
 and we rescale according to
 \begin{equation}\label{rescaling2}
 t=\frac{1}{(\Pr\Ra)^{\frac{1}{2}}}\hat t, \; u=(\Pr\Ra)^{\frac{1}{2}}\hat u,\; p=\Ra\hat p \; \mbox{ and thus }\; \Nu=(\Pr\Ra)^{\frac{1}{2}}\widehat{\Nu} 
 \end{equation}
 we end up with the system
 \begin{equation*}
 \left\{\begin{array}{rclc}
   \partial_{\hat t}  T+\hat u\cdot \nabla  T &=& 0 \,,\\
   \frac{1}{\Pr}(\partial_{\hat t}  \hat u+\hat u\cdot \nabla \hat u)+\nabla \hat p&=& \Ra T e_z  \,,\\
 \nabla\cdot \hat u&=& 0  \,.\\
 \end{array}\right.
\end{equation*}
Imitating the previous argument
we can conclude that   $\Nu\sim \Pr^{\frac{1}{2}}\Ra^{\frac{1}{2}}.$
\newline

From the equations of motion we can easily derive useful properties
and representations of the 
Nusselt number:
Testing the equation of the temperature  with $T$ and using 
(\ref{heatflux-property}) for $z=0$ we get 
\begin{equation}\label{N1}
 \Nu=\int_0^1\langle|\nabla T|^2\rangle dz,
\end{equation}
while testing  the Navier Stokes equations with $u$, using the incompressibility
and the boundary conditions for $u$, we have the energy inequality 
\begin{equation}\label{N2}
 \int_0^{1}\langle|\nabla u|^2\rangle dz\leq \Ra(\Nu-1)
\end{equation}
for Leray solutions.
\newline
Property (\ref{heatflux-property}) together with the boundary conditions
for $T$ and the maximum principle (\ref{Max-Princ})
yield
\begin{equation}\label{N3bis}
 \Nu\leq \frac{1}{\delta}\int_0^{\delta}\langle Tu^z\rangle dz+\frac{1}{\delta}\,,
\end{equation}
where the vertical average is taken over an arbitrary but
small boundary layer  of thickness $\delta$.
Moreover, using once more the maximum principle for the temperature (\ref{Max-Princ}) we get
\begin{equation}\label{N3}
 \Nu\leq \frac{1}{\delta}\int_0^{\delta}\langle|u^z|\rangle dz+\frac{1}{\delta}\,.
 \end{equation}
 
The first rigorous upper bound in the case $\Pr<\infty$ was derived in 1996 
by Constantin and Doering \cite{CD96}.
Exploiting the vanishing of the normal velocity at the boundaries
they could prove 
$$\Nu\lesssim\Ra^{\frac{1}{2}},$$
where with the symbol $\lesssim$  we denote the symbol $\leq$ up to
universal constants.
One hopes to get a better estimate exploiting 
the incompressibility condition which, in conjunction with the full no-slip boundary condition, implies 
 $$\partial_z u^z=0 \mbox{ at } z=0,1.$$
To illustrate this idea in a simpler situation, we first consider
the case of $\Pr=\infty$ as treated by Constantin and Doering (following the argument in \cite{Otto-Seis}).
Starting with inequality (\ref{N3}) and using Jensen's inequality in the form of 
$\left|\frac{d^2}{dz^2}\langle|u^z|\rangle\right|\leq \langle |\partial_z^2 u|\rangle$ we get
\begin{equation}\label{N4}
 Nu\leq \delta^2 \sup_z\langle|\partial_z^2 u^z|\rangle+\frac{1}{\delta}\,.
\end{equation}
Note that for $\Pr=\infty$, the velocity is instantaneously slaved to the temperature
via the stationary Stokes equations
\begin{equation*}
 \left\{\begin{array}{rclc}
   -\Delta u+\nabla  p&=& \Ra T e_z\,,\\
     \nabla\cdot u&=&0\,,\\
 \end{array}\right.
\end{equation*}
with the no-slip boundary condition.
Loosely speaking, the theory of maximal regularity states that for any "reasonable "
norm $||\cdot||$, one has 
$$||\nabla^2 u||\lesssim ||\Ra T e_z||.$$
However the Calder\'on-Zygmund theory fails for $||\cdot||=\sup_{z}\langle|\cdot|\rangle.$
Nevertheless, Constantin and Doering proved that, up to logarithms, one indeed has
$$\sup_z\langle|\partial_z^2 u^z|\rangle\lesssim \sup_z\langle|\Ra Te_z|\rangle\,.$$
By the maximum principle for the temperature (\ref{Max-Princ}),
the last bound implies
$$\sup_z\langle|\partial_z^2 u^z|\rangle\lesssim \Ra$$
and inserting this result into (\ref{N4}), we get
$$\Nu\lesssim \delta^2 \Ra+\frac{1}{\delta}\,.$$ 
The minimum occurs for $\delta\sim \Ra^{-\frac{1}{3}}$ and the resulting bound is, up to logarithmic corrections,
$\Nu\lesssim \Ra^{\frac{1}{3}}$.

In this paper we will consider the case $\Pr<\infty$.
\newline
We perturb the non-stationary Stokes equations, bringing {\it only} the nonlinear term to the right hand side
\begin{equation}
\label{non-stat}\frac{1}{\Pr}\partial_t u-\Delta u+\nabla p=\Ra T e_z -\frac{1}{\Pr}(u\cdot \nabla) u.
\end{equation}
\newline
As a side remark we now argue why \begin{equation}\label{perturb}\Pr\gg\Ra\end{equation} amounts to the classical
perturbative regime for Navier-Stokes equations. 
The classical perturbative argument goes as follows: One seeks a norm 
\newline
$||\cdot||$ in which the maximal regularity estimate
for the non-stationary Stokes equations (\ref{non-stat}) holds, yielding
$$||\nabla^2 u||\lesssim ||\Ra Te_z-\frac{1}{\Pr}(u\cdot \nabla)u||.$$
This norm has to be strong enough to control the convective nonlinearity 
in the sense of 
$$||(u\cdot \nabla)u||\lesssim ||\nabla^2u||^2.$$
In the application the norm should be also sufficiently weak 
so that (\ref{Max-Princ}) translates into
$$||\Ra Te_z||\lesssim \Ra.$$ 
\newline
The combination yields the estimate 
$$||\nabla^2 u||\lesssim\Ra+ \frac{1}{\Pr}||\nabla^2u||^2,$$
which is nontrivial only when  $\Pr\gg\Ra$.
\newline
Our analysis does not attempt to buckle via the classical perturbation
argument but instead uses the dissipation bound (\ref{N2}) to estimate 
the convective nonlinearity:
By the Cauchy-Schwarz inequality, Hardy's inequality and capitalizing once more on the no-slip boundary 
 condition we have
 \begin{equation}\label{Key}
  \int_0^1 \langle|(u\cdot \nabla)u|\rangle \frac{dz}{z} \lesssim \left(\int_0^1\langle|u|^2\rangle\frac{dz}{z^2}\right)^{\frac{1}{2}}\left(\int_{0}^{1}\langle|\nabla u|^2\rangle dz\right)^{\frac{1}{2}}
  \lesssim \int_0^1 \langle|\nabla u|^2\rangle dz ,
 \end{equation}
 which, using (\ref{N2}), implies
   \begin{equation}\label{Key-estim}\int_0^1 \langle|(u\cdot \nabla)u|\rangle \frac{dz}{z} \lesssim \Nu\Ra.\end{equation}        
It is this estimate that motivates the maximal regularity theory in the norm 
$||\cdot ||=\int_0^{\infty}\langle\cdot\rangle \frac{dz}{z}$.

In order to bound the right hand side of (\ref{N3}) we split the solution to the 
Navier-Stokes equations $u$ as 
$$u=u_{CD}+u_{NL}+u_{IV},$$ where 
$u_{CD}$ satisfies the non-stationary Stokes equations with the buoyancy force as right hand side 
\footnote{The stationary version of this problem was already analyzed by Constantin and Doering
in the seminal paper of 1999. This motivates the subscript CD.}
 \begin{equation}\label{M1}
  \left\{\begin{array}{rclc}
         \frac{1}{\Pr}\partial_t u_{CD}-\Delta u_{CD}+\nabla p_{CD}&=&\Ra T e_z \qquad & {\rm for } \quad 0<z<1 \,,\\
          \nabla\cdot u_{CD}&=& 0 \qquad & {\rm for } \quad 0<z<1\,,\\
           u_{CD}&=&0 \qquad & {\rm for } \quad z\in\{0,1\}\,,\\
           u_{CD}&=&0 \qquad & {\rm for } \quad t=0\,,\\
   \end{array}\right.
 \end{equation}
  $u_{NL}$ satisfies the non-stationary Stokes equations with the nonlinear term as right hand side
  \footnote{The subscript NL stands for non-linear. Indeed only in this
 equation the non-linear term of the Navier Stokes equations is appearing as 
 right hand side.}
 \begin{equation}\label{M2}
  \left\{\begin{array}{rclc}
       \frac{1}{\Pr}\partial_t u_{NL}-\Delta u_{NL}+\nabla p_{NL}&=&-\frac{1}{\Pr}(u\cdot \nabla) u \qquad & {\rm for } \quad 0<z<1 \,,\\
      \nabla\cdot u_{NL}&=& 0 \qquad & {\rm for } \quad 0<z<1\,.\\
       u_{NL}&=&0 \qquad & {\rm for } \quad z\in\{0,1\}\,,\\
      u_{NL}&=&0 \qquad & {\rm for } \quad t=0\,\\
     \end{array}\right.
 \end{equation} 
    and $u_{IV}$ satisfies the
    non-stationary Stokes equations with zero forcing term and non-zero initial values
    \footnote{ The subscript IV stands for initial value. }
     \begin{equation}\label{M3}
       \left\{\begin{array}{rclc}
               \frac{1}{\Pr}\partial_t u_{IV}-\Delta u_{IV}+\nabla p_{IV}&=& 0 \qquad & {\rm for } \quad 0<z<1 \,,\\
        \nabla\cdot u_{IV}&=& 0 \qquad & {\rm for } \quad 0<z<1\,,\\
         u_{IV}&=&0 \qquad & {\rm for } \quad z=0\,,\\
         u_{IV}&=&u_0 \qquad & {\rm for } \quad t=0\,.\\
        \end{array}\right.
     \end{equation} 
Inserting the decomposition $u=u_{CD}+u_{NL}+u_{IV}$ into the bound (\ref{N3}) for the Nusselt number, we have 
\begin{equation}\label{N5}
\begin{array}{rclc}
 \Nu&\leq& \frac{1}{\delta}\int_0^{\delta}\langle|u^z|\rangle dz+\frac{1}{\delta}\\
&\leq& \sup_{z\in(0,\delta)}\langle|u_{CD}^z|\rangle+
\delta\int_0^{\delta}\langle|\partial_z^2 u_{NL}^z|\rangle dz+\delta^{-\frac{1}{2}}\left(\int_0^{\delta}\langle|u_{IV}^z|^2\rangle dz\right)^{\frac{1}{2}}+\frac{1}{\delta}\\
&\leq& \delta^2 \left( \sup_{z\in(0,\delta)}\langle|\partial_z^2 u_{CD}^z|\rangle+ 
\int_0^{\delta}\langle|\partial_z^2 u_{NL}^z|\rangle \frac{dz}{z}\right)+\delta^{-\frac{1}{2}}\left(\int_0^{\delta}\langle|u_{IV}^z|^2\rangle dz\right)^{\frac{1}{2}}+\frac{1}{\delta}.
\end{array}
\end{equation}
%
%
We notice that  
\begin{equation}\label{N67}
\int_0^{\delta}\langle|u_{IV}^z|^2\rangle dz=0\,.
\end{equation}
Indeed testing the equation (\ref{M3}) with $u_{IV}$ we find that 
$$\int_0^{t_0}\int_{x'}\int_{z}|\nabla u_{IV}(x',z,t)|^2 dzdx'dt\leq \int_{x'}\int_{z}|u_0(x',z)|^2 dzdx'$$
and in turn by the Poincar\'e inequality and passing to limits we get
 $$\int_0^1\langle|u_{IV}|^2\rangle dz= 0.$$
\newline
On the one hand, for equation (\ref{M1}) we expect the following logarithmically failing maximal regularity bound
\begin{equation}\label{N6}
\sup_{z\in(0,1)}\langle|\partial_z^2 u_{CD}^z|\rangle\lesssim \Ra,
\end{equation}
just as for the case of $\Pr=\infty$.
\newline
On the other hand, the problem of bounding the term $\int_0^{1}\langle|\partial_z^2 u_{NL}^z|\rangle \frac{dz}{z}$ in (\ref{N5})
requires new techniques.
Nevertheless  we expect    
\begin{equation}\label{NM}
\int_0^{\delta}\langle|\partial_z^2 u_{NL}^z|\rangle \frac{dz}{z}\lesssim
\frac{1}{\Pr}\int_0^1 \langle|(u\cdot \nabla)u|\rangle \frac{dz}{z}\,,
\end{equation}
 up to logarithmic corrections.
Using (\ref{Key-estim}) we obtain
 \begin{equation}\label{N7}
  \int_0^{\delta}\langle|\partial_z^2 u_{NL}^z|\rangle \frac{dz}{z}\lesssim\frac{1}{\Pr}\Nu\Ra\,.
 \end{equation}
 Inserting  (\ref{N67}), (\ref{N6}) and (\ref{N7}) into the bound (\ref{N5}) for the Nusselt number
 and ignoring logarithmic correction factors, we  get 
 \begin{eqnarray*}
 \Nu
  \lesssim \delta^2 \Ra(1+\frac{1}{\Pr} \Nu)+\frac{1}{\delta}.
\end{eqnarray*}
After choosing $\delta\sim \left(\Ra(1+\frac{\Nu}{\Pr})\right)^{-\frac{1}{3}}$ 
and applying Young's inequality, we have
$$\Nu\approx \Ra^{\frac{1}{3}}+\left(\frac{\Ra}{\Pr}\right)^{\frac{1}{2}}\,,$$
which implies, up to logarithms, 

\begin{equation}\label{N-EST}
         \Nu\lesssim \begin{cases}
            \Ra^{\frac{1}{3}} & \mbox{ for } \Pr\geq \Ra^{\frac{1}{3}}\,,\\
            \Pr^{-\frac{1}{2}}\Ra^{\frac{1}{2}} & \mbox{ for } \Pr\leq \Ra^{\frac{1}{3}}\,.\\     
           \end{cases}
\end{equation}

\subsection{Notations}\label{notations}

The $(d-1)-$dimensional torus:
\vspace{0.1cm}\newline
We denote with $[0,L)^{d-1}$  the $(d-1)-$dimensional torus of lateral size $L$.
\newline
The spatial vector:

$$x=(x',z)\in [0,L)^{d-1}\times \R \,.$$
The velocity vector field:

$$u=(u',u^z)\in \R^d  \mbox{ where } u'\in \R^{d-1}  \mbox{ and } u^z\in \R .$$
The horizontal average:

\begin{equation*}
   \langle\cdot\rangle'=\frac{1}{L^{d-1}}\int_{[0,L)^{d-1}} \;\cdot\;\;\; dx'\,.
\end{equation*}
Long-time and horizontal average:
\begin{equation}\label{LTaHA}
  \langle\cdot\rangle= \limsup_{t_0\rightarrow \infty}\frac{1}{t_0}\int_0^{t_0}\langle \;\cdot \;\rangle'  dt\,.
\end{equation}
Convolution in the horizontal direction:

\begin{equation*}
 f\ast_{x'}g(x')=\int_{[0,L)^{d-1}}f(x'-\widetilde{x'})g(\widetilde{x'})d\widetilde{x'}\;.
\end{equation*}
Convolution in the whole space:

\begin{equation*}
 f\ast g(x)=\int_{\R}\int_{[0,L)^{d-1}}f(x'-\widetilde{x'},z-\tilde{z})g(\widetilde{x'},\tilde{z})d\widetilde{x'}d\tilde{z}\,.
\end{equation*}
Horizontally band-limited function:
\newline
A function $g=g(x',z,t)$ is called {\it horizontally  band-limited} with bandwidth $R$ if it satisfies 
the {\it bandedness assumption }
\begin{equation}\label{BANDLIM2}
 \F g(k',z,t)=0 \mbox{ unless } 1\leq R|k'|\leq 4 \mbox{ where } R<R_0.
\end{equation}
Interpolation norm:
$$||f||_{(0,1)}=||f||_{R;(0,1)}=\inf_{f=f_1+f_2}\left\{\sup_{z\in(0,1)}\langle|f_1|\rangle+\int_{(0,1)}\langle|f_2|\rangle\frac{dz}{z(1-z)}\right\}\,,$$
$$||f||_{(0,\infty)}=||f||_{R;(0,\infty)}=\inf_{f=f_1+f_2}\left\{\sup_{z\in(0,\infty)}\langle|f_1|\rangle+\int_{(0,\infty)}\langle|f_2|\rangle\frac{dz}{z}\right\}\,,$$
$$||f||_{(-\infty,1)}=||f||_{R;(-\infty,1)}=\inf_{f=f_1+f_2}\left\{\sup_{z\in(-\infty,1)}\langle|f_1|\rangle+\int_{(-\infty,1)}\langle|f_2|\rangle\frac{dz}{1-z}\right\}\,.$$
where $f_0, f_1$ satisfy the bandedness assumption (\ref{BANDLIM2}).
\newline
Horizontal Fourier transform:

$$\mathcal{F'}f(k',z,t)=\int e^{-ik'\cdot x'}f(x',z,t)dx'\,.$$
where $k'$ is the dual variable of $x'$.
\vspace{0.3cm}\newline
Throughout the paper we will denote with $\lesssim$ the inequality up to universal constants.

\section{Proof of Theorem \ref{th2}}

Without loss of generality we will assume $u_0=0$ since we have already seen that the contribution of $u_{IV}$
to the Nusselt number is zero (see (\ref{N67})).

Let us fix a smooth cut-off function $\psi$ in Fourier space satisfying
$$\psi (k')=\begin{cases} 1& 0\leq |k'|\leq \frac{7}{2}\,, \\ 0 & \qquad |k'|\geq 4 \,.\end{cases}$$
Consider the function $\zeta (k')=\psi(k')-\psi(\frac{7}{2}k')$ and define $\zeta_{j}(k')=\zeta (2^{-j}k')$.
Notice that $\zeta_{j}$ is supported in $(2^{j}, 2^{j+2})$.
\newline
Following a Littlewood-Paley-type decomposition  we construct
three operators $\mathbb{P_{<}}$, $\mathbb{P}_j$ and $\mathbb{P_>},$ which act at the
level of Fourier space by multiplication by cut off functions that localize to 
small, intermediate and large wavelengths respectively:
$$ \F \mathbb{P}_{<}f=\zeta_{<}\F f,$$
$$ \F \mathbb{P}_jf=\zeta_{j}\F f,$$
$$ \F \mathbb{P}_{>}f=\zeta_{>}\F f,$$
where $\zeta_{<}=\sum_{j<j_1}\zeta_{j}$ and $\zeta_{>}=\sum_{j>j_2}\zeta_{j} $ with $j_1<j_2$ to be determined. 
We notice that the operator $\mathbb{P}_j:L^p\rightarrow L^p$ for $1\leq p<\infty$ is bounded.
   \newline
   Inserting the decomposition into the bound for the Nusselt number (\ref{N3bis}) we get
   \begin{equation}\label{NUU}
   \begin{array}{rclc}
    \Nu&\leq&\frac{1}{\delta}\int_0^{\delta}\langle Tu^z\rangle dz+\frac{1}{\delta}\\
      &=&\frac{1}{\delta}\int_0^{\delta}\langle T\mathbb{P_{<}}u^z\rangle dz+\sum_{j=j_1}^{j_2}\frac{1}{\delta}\int_0^{\delta}\langle T\mathbb{P}_ju^z\rangle dz+\frac{1}{\delta}\int_0^{\delta}\langle T\mathbb{P_>}u^z\rangle+\frac{1}{\delta}\,.
      \end{array}    
      \end{equation}  
      At first, let us focus on the second term in (\ref{NUU}) arising
       from the intermediate wavelengths. In order to bound this 
        term we will need the maximal regularity estimate stated 
         in Proposition \ref{prop3}.
          For this purpose, rewrite the Navier-Stokes equations in (\ref{NSE}) as non-stationary Stokes equations with the
           nonlinear term and the buoyancy term in the right hand side
           \begin{equation}\label{PSE}
         \left\{\begin{array}{rclc} 
         \frac{1}{\Pr}\partial_t \mathbb{P}_ju-\Delta \mathbb{P}_ju+\nabla \mathbb{P}_jp&=&\Ra \mathbb{P}_jT e_z- \frac{1}{\Pr}\mathbb{P}_j(u\cdot \nabla) u  \qquad & {\rm for } \quad 0<z<1 \, ,\\
         \nabla\cdot \mathbb{P}_ju&=&0  \qquad & {\rm for } \quad 0<z<1 \,,\\
         \mathbb{P}_ju&=&0  \qquad & {\rm for } \quad z\in\{0,1\}\,,\\
         \mathbb{P}_ju&=&0  \qquad & {\rm for } \quad t=0\,.\\
         \end{array}\right.
         \end{equation}
         Observe that the application of the operator ``horizontal filtering``, namely $\mathbb{P}_j$, preserves the no-slip boundary condition
         at $z=0,1$ and it commutes with $\partial_z $ and all the differential operators that act in the vertical direction.                         
         Using the maximum  principle for the temperature (\ref{Max-Princ}), the Poincar\'e inequality in the $z-$variable and considering a generic decomposition
         of $\partial_z^2 \mathbb{P}_ju^z=h_0+h_1$ we have 
         \begin{eqnarray*}
          \frac{1}{\delta}\int_0^{\delta}\langle T\mathbb{P}_ju^z\rangle dz &\leq& \frac{1}{\delta}\int_0^{\delta}\langle|\mathbb{P}_ju^z|\rangle dz\\
           &\leq&\delta \int_0^{\delta}\langle|\partial_z^2 \mathbb{P}_ju^z|\rangle dz \\
           &\leq&\delta \left( \int_0^{\delta}\langle|h_0|\rangle dz+\int_0^{\delta}\langle|h_1|\rangle dz \right)\\
           &\leq& \delta^2\left(\sup_{0<z<1}\langle|h_0|\rangle+\int_0^1\langle|h_1|\rangle\frac{dz}{z(1-z)}\right)\,.
           \end{eqnarray*}
           Passing to the infimum over all the possible decompositions 
           of  $\partial_z^2 \mathbb{P}_ju^z$ we get 
           \begin{eqnarray*}
           \frac{1}{\delta}\int_0^{\delta}\langle T\mathbb{P}_ju^z\rangle dz \leq\delta^2 ||\partial_z^2 \mathbb{P}_ju^z||_{(0,1)}\,.
           \end{eqnarray*}
            We notice that $\mathbb{P}_ju$ satisfies the linear Stokes equations (\ref{PSE}) and 
            for $j>j_1$ it satisfies the bandedness assumption provided 
            \begin{equation}\label{J1}
             j_1\gtrsim \log_2 R_0^{-1}\,.      
            \end{equation}
             Therefore by the maximal regularity estimate (\ref{MRE}) applied to $\mathbb{P}_ju$  we have 
           \begin{eqnarray*}
           ||\partial_z^2 \mathbb{P}_ju^z||_{(0,1)} \leq \sup_{0<z<1}\langle|\Ra \mathbb{P}_jTe_z|\rangle+\int_0^1\langle|\frac{1}{\Pr}\mathbb{P}_j(u\cdot \nabla) u |\rangle\frac{dz}{z(1-z)}\,.
           \end{eqnarray*}
            Applying again the maximum principle (\ref{Max-Princ}) to the first term of the right 
            hand side we find
            $$\sup_{0<z<1}\langle|\Ra \mathbb{P}_jTe_z|\rangle\leq \Ra\,.$$
            To estimate the nonlinear  part we apply the Cauchy-Schwarz inequality and  Hardy's 
            inequality 
            \begin{eqnarray*}
            &&\int_0^1\langle|\frac{1}{\Pr}\mathbb{P}_j(u\cdot \nabla) u |\rangle\frac{dz}{z(1-z)}\\
            &\leq& \frac{1}{\Pr}\int_0^1\langle|\mathbb{P}_j(u\cdot \nabla) u |\rangle\frac{dz}{z}+\frac{1}{\Pr}\int_0^1\langle|\mathbb{P}_j(u\cdot \nabla) u |\rangle\frac{dz}{1-z}\\
            &\lesssim&\frac{1}{\Pr}\left(\int_0^1\frac{1}{z^2}\langle|\mathbb{P}_ju|^2\rangle dz  + \int_0^1\frac{1}{(1-z)^2}\langle|\mathbb{P}_ju|^2\rangle dz\right)^{\frac{1}{2}} \left(\int_0^1\langle| \nabla \mathbb{P}_ju|^2\rangle dz\right)^{\frac{1}{2}}\\
            &\lesssim&\frac{1}{\Pr}\left(\int_0^1\langle|\partial_z\mathbb{P}_ju|^2\rangle dz\right)^{\frac{1}{2}}\left(\int_0^1\langle| \nabla \mathbb{P}_ju|^2\rangle dz\right)^{\frac{1}{2}}\\
            &\lesssim&\frac{1}{\Pr}\int_0^1\langle|\nabla u|^2\rangle dz\\
            &\stackrel{(\ref{N2})}{\leq}&\frac{1}{\Pr}\Ra(\Nu-1)\,.
            \end{eqnarray*}                                                    
            Summing up over all the intermediate wavelengths we obtain  
            \begin{equation}\label{IWL}
                        \sum_{j=j_1}^{j_2}\frac{1}{\delta}\int_0^{\delta}\langle T\mathbb{P}_ju^z\rangle dz\lesssim (j_2-j_1)\delta^2 \left(\Ra +\frac{1}{\Pr}\Ra(\Nu-1)\right)\,.
            \end{equation}
            We now turn to the first term appearing on the right hand side of 
            (\ref{NUU}), contribution of the small wavelengths. By using the Cauchy-Schwarz inequality,
            the divergence-free condition, the horizontal bandedness assumption in form of (\ref{BAND2}) and the Poincar\'e inequality in the $z-$variable, 
            we obtain
            \begin{align}
            \frac{1}{\delta}\int_0^{\delta}\langle  T\mathbb{P}_< u^z \rangle dz 
             =&\;\frac{1}{\delta}\int_0^{\delta}\langle (T-1)\mathbb{P}_< u^z \rangle dz\numberthis\label{zero-av}\\ 
             \lesssim& \;\frac{1}{\delta} \left(\int_0^{\delta}\langle| T-1|^2\rangle dz\right)^{\frac{1}{2}}\left(\int_0^{\delta}\langle|\mathbb{P}_{<}u^z|^2\rangle dz\right)^{\frac{1}{2}}\notag\\
             \lesssim&\;\delta\frac{1}{\delta} \left(\int_0^{\delta}\langle| \partial_z T|^2\rangle dz\right)^{\frac{1}{2}}\delta\left(\int_0^{\delta}\langle|\partial_z\mathbb{P}_{<}u^z|^2\rangle dz\right)^{\frac{1}{2}}\notag\\
             \lesssim&\;\delta  \left(\int_0^{\delta}\langle|  \nabla T|^2\rangle dz\right)^{\frac{1}{2}}\left(\int_0^{\delta}\langle|\mathbb{P}_{<}\nabla'\cdot u'|^2\rangle dz\right)^{\frac{1}{2}}\notag\\
             \lesssim&\;\delta  \left(\int_0^{\delta}\langle| \nabla  T|^2\rangle dz\right)^{\frac{1}{2}}2^{j_1}\left(\int_0^{\delta}\langle|u'|^2\rangle dz\right)^{\frac{1}{2}}\notag\\
             \lesssim&\;\delta  \left(\int_0^{\delta}\langle| \nabla  T|^2\rangle dz\right)^{\frac{1}{2}}2^{j_1}\delta\left(\int_0^{\delta}\langle|\partial_z u'|^2\rangle dz\right)^{\frac{1}{2}}\notag\\
             \lesssim&\; 2^{j_1} \delta^2 \left(\int_0^{\delta}\langle| \nabla  T|^2\rangle dz\right)^{\frac{1}{2}}\left(\int_0^{\delta}\langle|\nabla' u|^2\rangle dz\right)^{\frac{1}{2}}\notag\\
             \lesssim& \;2^{j_1}\delta^2 \left(\int_0^{1}\langle| \nabla  T|^2\rangle dz\right)^{\frac{1}{2}}\left(\int_0^{1}\langle|\nabla u|^2\rangle dz\right)^{\frac{1}{2}}\notag\\
             \stackrel{(\ref{N1})\&(\ref{N2})}{\lesssim}&\; 2^{j_1}\delta^2 \Ra^{\frac{1}{2}}\Nu\,,\numberthis\label{SWL}  
             \end{align}
             where in (\ref{zero-av}) we used the fact that $\langle u^z \rangle=0$.
             Finally, we turn to the third term in (\ref{NUU}), which represents
             the contribution from the large wavelengths. In order to estimate
             this term we use the Cauchy-Schwarz inequality, the Poincar\'e inequality
             in the $z-$variable
             and the horizontal bandedness assumption in form of (\ref{BAND1}) applied to $T$
             \begin{align}
             \frac{1}{\delta}\int_0^{\delta}\langle \mathbb{P}_>  T u^z \rangle dz  \leq&\;\frac{1}{\delta} \left(\int_0^{\delta}\langle| \mathbb{P}_{>} T|^2\rangle dz\right)^{\frac{1}{2}}\left(\int_0^{\delta}\langle|u^z|^2\rangle dz\right)^{\frac{1}{2}}\notag\\
              \leq&\;\frac{1}{\delta} \frac{1}{2^{j_2}}\left(\int_0^{\delta}\langle | \nabla'T|^2 \rangle dz\right)^{\frac{1}{2}}\delta\left(\int_0^{\delta}\langle|\partial_zu^z|^2\rangle dz\right)^{\frac{1}{2}}\notag\\
              \leq& \;\frac{1}{2^{j_2}}\left(\int_0^{1}\langle |\nabla T|^2 \rangle dz\right)^{\frac{1}{2}}\left(\int_0^{\delta}\langle|\nabla u|^2\rangle dz\right)^{\frac{1}{2}}\notag\\
              \stackrel{(\ref{N1})\&(\ref{N2})}{\lesssim}& \;\frac{1}{2^{j_2}}\Ra^{\frac{1}{2}}\Nu\,.  \numberthis\label{LWL}
              \end{align}
               Putting the three estimates (\ref{IWL}),(\ref{SWL}) and (\ref{LWL}) together, we have the following bound on the
               Nusselt number
             
                $$ \Nu \lesssim (j_2-j_1)\delta^2\left(\frac{ \Nu}{\Pr}+1\right)\Ra+\left(\delta^2 2^{j_1}+\frac{1}{2^{j_2}}\right)\Ra^{\frac{1}{2}}\Nu+ \frac{1}{\delta}\,. $$                 
                In the last inequality  we impose  $2^{-j_2}=2^{j_1}\delta^2$. In turn, observe that $2^{-j_2}=2^{-\frac{(j_2-j_1)}{2}}\delta$ and therefore 
                
                \begin{equation}\label{Opt1}\Nu \lesssim  (j_2-j_1)\delta^2\left(\frac{\Nu}{\Pr}+1\right)\Ra+2^{-\frac{(j_2-j_1)}{2}}\delta \Ra^{\frac{1}{2}}\Nu+ \frac{1}{\delta}\,. \end{equation}              
                Observe that, on the one hand, we want the second term of the right hand side to be absorbed in the left hand side, therefore we impose
                $$ 1\approx 2^{-\frac{(j_2-j_1)}{2}}\delta \Ra^{\frac{1}{2}}$$
                and, on the other hand, we require all the terms in the right hand side to be of the same size
                \begin{equation}\label{DELTA-CHOICE}(j_2-j_1)\left(\frac{\Nu}{\Pr}+1\right)\Ra\approx \frac{1}{\delta^3}.\end{equation}
                From these two conditions we deduce 
               $$(j_2-j_1)2^{\frac{3}{2}(j_2-j_1)}\approx \Ra^{\frac{1}{2}}\left(\frac{\Nu}{\Pr}+1\right)^{-1},$$
               which is of the form 
               $x\log_{a}x\approx y$ with $x=a^{(j_2-j_1)}$ and $y=\Ra^{\frac{1}{2}}\left(\frac{\Nu}{\Pr}+1\right)^{-1}$ for $a>1$.
               This implies that, asymptotically, $x\approx \frac{y}{\log_a y}$ and therefore
               $$j_2-j_1\approx \log_a \left(\frac{\Ra^{\frac{1}{2}}\left(\frac{\Nu}{\Pr}+1\right)^{-1}}{\log_a(\Ra^{\frac{1}{2}}\left(\frac{\Nu}{\Pr}+1\right)^{-1})}\right)\approx \ln \Ra.$$
               Inserting this back into (\ref{DELTA-CHOICE}), we are led to the natural choice of $\delta$
               $$\delta=\left(\left(\frac{\Nu}{\Pr}+1\right)\Ra\ln \Ra\right)^{-\frac{1}{3}},$$               
%
                which give us the bound 
               $$\Nu\lesssim  \left(\left(\frac{\Nu}{\Pr}+1\right) \Ra\ln\Ra\right)^{\frac{1}{3}}\,.$$
               Applying the triangle inequality \footnote{Note that for $0<p<1$ we have
               $$||f+g||_{p}\leq 2^{\frac{1}{p}-1}(||f||_{p}+||g||_{p})$$ }
               $$\Nu\lesssim   \left(\Ra\ln\Ra\right)^{\frac{1}{3}}+\left(\left(\frac{\Nu}{\Pr}\right) \Ra\ln\Ra\right)^{\frac{1}{3}}\,$$
               and Young's inequality, we finally obtain
               $$\Nu\lesssim (\Ra\ln \Ra)^{\frac{1}{3}}+\left(\frac{\Ra\ln \Ra}{\Pr}\right)^{\frac{1}{2}}\,.$$
               In conclusion we get the following bound on the Nusselt number
               $$\Nu \lesssim \begin{cases}
               (\Ra\ln \Ra)^{\frac{1}{3}} & \mbox{ for } \Pr\geq (\Ra\ln \Ra )^{\frac{1}{3}}\,,\\
               \left( \frac{\Ra}{\Pr}\ln \Ra\right)^{\frac{1}{2}} & \mbox{ for } \Pr\leq (\Ra\ln\Ra )^{\frac{1}{3}}\,.
                \end{cases}$$

\section{Maximal regularity in the strip}

\subsection{From the strip to the half space}
Let us consider the non-stationary Stokes equations
\begin{equation*}
  \left\{\begin{array}{rclc}
     \partial_t u-\Delta u+\nabla p &=& f \qquad & {\rm for } \quad 0<z<1 \,,\\
        \nabla\cdot u &=& 0 \qquad & {\rm for }  \quad 0<z<1\,,\\
         u &=& 0  \qquad & {\rm for } \quad z\in\{0,1\} \,,\\
         u &=& 0  \qquad  & {\rm for } \quad t=0 \,.\\
         \end{array}\right.        
 \end{equation*}
In order to prove the maximal regularity 
estimate in the strip we extend the problem 
(\ref{STOKES-STRIP}) in the half space.
By symmetry, it is enough to consider for the moment
the extension to the upper half space.
\newline
Consider the localization $(\tilde u, \tilde p):=(\eta u,\eta p)$ where
 
 \begin{equation}\label{cutoff} \eta(z) \mbox{ is a cut-off function for } [0,\frac{1}{2}) \mbox{ in } [0,1) \,.\end{equation}
 Extending $(\tilde u, \tilde p)$ by zero they can be viewed as functions in the upper half space.
The couple  $(\tilde u, \tilde p)$ satisfies 

 \begin{equation}\label{UHS}
  \left\{\begin{array}{rclc}
      \partial_t \tilde u-\Delta \tilde u+\nabla \tilde p &=& \tilde f \qquad & {\rm for } \quad z>0\,,\\
        \nabla\cdot \tilde u &=& \tilde\rho \qquad & {\rm for } \quad z>0 \,,\\
        \tilde u &=&  0 \qquad & {\rm for } \quad z=0\,,\\
         \tilde u &=& 0  \qquad  & {\rm for } \quad t=0 \,,\\
         \end{array}\right.        
 \end{equation}
where 
\begin{equation}\label{Defi}
\tilde f:=\eta f-2(\partial_z \eta)\partial_z u-(\partial_z^2\eta )u+(\partial_z\eta )pe_z, \qquad \qquad  \tilde\rho:=(\partial_z\eta )u^z\,.
\end{equation}
%

\subsection{Maximal regularity in the upper half space}
In the half space, taking 
advantages from the explicit representation of the solution 
via Green functions, we prove 
the regularity estimates which will be crucial in the proof of
Theorem \ref{th1}.
\begin{proposition}[Maximal regularity in the upper half space]\label{pr1}\ \\
Consider the non-stationary Stokes equations in the upper half-space 
 \begin{equation}\label{STOKES-HALF}
  \left\{\begin{array}{rclc}
      \partial_t u-\Delta u+\nabla p &=& f \qquad & {\rm for } \quad z>0\,,\\
         \nabla\cdot u &=& \rho \qquad & {\rm for } \quad z>0 \,,\\
         u &=&  0 \qquad & {\rm for } \quad z=0\,,\\
         u &=&  0 \qquad & {\rm for } \quad t=0\,.\\
         \end{array}\right.        
 \end{equation}
 Suppose that $f$ and $\rho$ are horizontally band-limited , i.e
\begin{equation}\label{BC1}\F f(k',z,t)=0 \mbox{ unless } 1\leq R|k'|\leq 4  \mbox{ where } R\in(0,\infty)\,,\end{equation}
and 
\begin{equation}\label{BC2}\F \rho(k',z,t)=0 \mbox{ unless } 1\leq R|k'|\leq 4  \mbox{ where } R\in(0,\infty)\,.\end{equation}
Then
\begin{eqnarray*}
 &&||\partial_t u^z||_{(0,\infty)}+||\nabla^2 u^z||_{(0,\infty)}+||\nabla p||_{(0,\infty)}+||(\partial_t -\partial_z^2)u'||_{(0,\infty)}+||\nabla'\nabla u'||_{(0,\infty)}\\
 &\lesssim&||f||_{(0,\infty)}+||(-\Delta')^{-\frac{1}{2}}\partial_t \rho||_{(0,\infty)}+||(-\Delta')^{-\frac{1}{2}}\partial_z^2 \rho ||_{(0,\infty)}+||\nabla \rho||_{(0,\infty)},
\end{eqnarray*}
where $||\cdot||_{(0,\infty)}$ denotes the norm
\begin{equation}\label{NORM-HALF}||f||_{(0,\infty)}:=||f||_{R;(0,\infty)}\inf_{f=f_0+f_1}\left(\sup_{0<z<\infty}\langle |f_0|\rangle+\int_0^{\infty}\langle |f_1|\rangle\frac{dz}{z}\right)\,,\end{equation}
where $f_0$ and $f_1$ satisfy the bandedness assumption (\ref{BC1}).
\end{proposition}
The first ingredient to establish Proposition \ref{pr1} is a suitable representation 
of the solution operator $(f=(f',f^z),\rho)\rightarrow u=(u',u^z) $
of the Stokes equations with the no-slip  boundary condition.
In the case of no-slip boundary condition the Laplace operator  
has to be factorized as
$\Delta=\partial_z^2+\Delta'=(\partial_z+(-\Delta')^{\frac{1}{2}})(\partial_z-(-\Delta')^{\frac{1}{2}})$.
In this way the solution operator  to the Stokes  equations with the no-slip boundary
 condition (\ref{STOKES-HALF}) can be written as the fourfold composition of solution operators 
to three more elementary boundary value problems: 
\begin{itemize}

\item Backward fractional diffusion equation (\ref{FraBack}):

\begin{equation}\label{FraBack}
  \left\{\begin{array}{rclc}
     (\partial_z-(-\Delta')^{\frac{1}{2}})\phi &=& \nabla\cdot f-(\partial_t-\Delta )\rho \qquad & {\rm for } \quad  z>0 \,,\\
        \phi &\rightarrow& 0  \qquad & {\rm for } \quad z\rightarrow \infty.\\
         \end{array}\right.        
 \end{equation}

\item  Heat equation (\ref{Heat1}):
\begin{equation}\label{Heat1}
  \left\{\begin{array}{rclc}
  (\partial_t-\Delta)v^z&=&(-\Delta')^{\frac{1}{2}}(f^z-\phi)-\nabla'\cdot f'+(\partial_t-\Delta)\rho \qquad & {\rm for } \quad  z>0,\\
  v^z&=& 0  \qquad & {\rm for } \quad  z=0 \,,\\
  v^z&=& 0  \qquad & {\rm for } \quad  t=0 \,.\\
  \end{array}\right. 
\end{equation}

\item  Forward fractional diffusion equation (\ref{FraFor}): 

\begin{equation}\label{FraFor}
  \left\{\begin{array}{rclc}
  (\partial_z+(-\Delta')^{\frac{1}{2}})u^z&=& v^z  \qquad & {\rm for } \quad   z>0\,,\\
  u^z&=&0   \qquad & {\rm for } \quad  z=0\,.\\
  \end{array}\right. 
\end{equation}

\item Heat equation (\ref{Heat2}):

\begin{equation}\label{Heat2}
 \left\{\begin{array}{rclc}
(\partial_t-\Delta)v'&=&(1+\nabla'(-\Delta')^{-1}\nabla'\cdot)f'   \qquad & {\rm for } \quad   z>0\,,\\
v'&=& 0   \qquad & {\rm for } \quad z=0\,,\\
v'&=& 0   \qquad & {\rm for } \quad  t=0\,.\\
  \end{array}\right.
\end{equation}
Finally set
\begin{equation}
\label{horvel}u'=v'-\nabla'(-\Delta')^{-1}(\rho-\partial_z u^z)\,.
\end{equation}

\end{itemize}
In order to prove the validity of the decomposition we need to argue that 
$$(\partial_t-\Delta)u-f \mbox{ is irrotational }\,,$$ which 
reduces to prove that
 \begin{equation*}
  (\partial_t-\Delta)u'-f' \mbox{ is irrotational in } x'
 \end{equation*}
 and
 \begin{equation}\label{SEC}
  \partial_z ((\partial_t-\Delta)u'-f')=\nabla'((\partial_t-\Delta)u^z-f^z)\,.
 \end{equation} 
 Let us consider for simplicity $\rho=0.$
The first statement follows easily from the definition. Indeed by definition (\ref{horvel}) and  equation (\ref{Heat2}), 
 $$ (\partial_t-\Delta)u'-f' =\nabla' ((-\Delta')^{-1}\nabla'\cdot f'+(-\Delta')^{-1}\partial_z u^z).$$
Let us now focus on (\ref{SEC}), which by using (\ref{horvel}) and (\ref{Heat2}) can be rewritten as
  $$\partial_z \nabla'((-\Delta')^{-1}\nabla'\cdot f'+(-\Delta')^{-1}(\partial_t-\Delta)\partial_zu^z)=\nabla'((\partial_t-\Delta)u^z-f^z)\,.$$
Because of the periodic boundary conditions in the horizontal direction, the latter is equivalent to
$$\partial_z (-\Delta')((-\Delta')^{-1}\nabla'\cdot f'+(-\Delta')^{-1}(\partial_t-\Delta)\partial_zu^z)=(-\Delta')((\partial_t-\Delta)u^z-f^z),$$
that, after factorizing  $\Delta=(\partial_z-(-\Delta')^{\frac{1}{2}})(\partial_z+(-\Delta')^{\frac{1}{2}})$, turns into
$$(\partial_z-(-\Delta')^{\frac{1}{2}})(\partial_t-\Delta)(\partial_z+(-\Delta')^{\frac{1}{2}}) u^z=(-\Delta')f^z-\partial_z\nabla'\cdot f'\,.$$
One can easily check that the identity holds true by applying (\ref{FraFor}), (\ref{Heat1}) and (\ref{FraBack}).
The no-slip boundary condition is trivially satisfied, indeed by (\ref{FraFor}) we have $u^z=0$ and $\partial_z u^z=0$. 
The combination of (\ref{horvel}) with $\partial_z u^z=0$ gives $u'=0$.
%
%
%
%
%
%

For each step of the decomposition of the Navier Stokes equations 
we will derive maximal regularity-type estimates. These are summed 
up in the following

\begin{proposition}\label{prop3}\ \\ 
 \begin{enumerate} 
  \item Let $\phi,f,\rho$ satisfy the problem (\ref{FraBack}) and assume
        $f,\rho$ are  horizontally band-limited, i.e 
        $$\F f(k',z,t)=0 \mbox{ unless } 1\leq R|k'|\leq 4$$
        and 
        $$\F \rho(k',z,t)=0 \mbox{ unless } 1\leq R|k'|\leq 4.$$
        
        Then, 
        
        \begin{equation*}
          ||\phi||_{(0,\infty)}\lesssim ||f||_{(0,\infty)}+||(-\Delta')^{-\frac{1}{2}}\partial_t \rho||_{(0,\infty)}+||\nabla \rho||_{(0,\infty)} \,.\label{A}
        \end{equation*}
        
  \item  Let $v^z, f, \phi, \rho$ satisfy the problem (\ref{Heat1}) and assume
          $f,\phi,\rho$ are  horizontally band-limited, i.e 
          $$\F f(k',z,t)=0 \mbox{ unless } 1\leq R|k'|\leq 4\,,$$         
          $$\F \phi(k',z,t)=0 \mbox{ unless } 1\leq R|k'|\leq 4\,$$
          and
          $$\F \rho(k',z,t)=0 \mbox{ unless } 1\leq R|k'|\leq 4\,.$$
          
          Then,
         \begin{equation}\label{B}
         \begin{array}{rclc}
          &&||\nabla v^z||_{(0,\infty)}+||(-\Delta)^{-\frac{1}{2}}(\partial_t-\partial_z^2)v^z||_{(0,\infty)}\\
          &\lesssim& ||f||_{(0,\infty)}+||\phi||_{(0,\infty)}+||(-\Delta')^{-\frac{1}{2}}\partial_t\rho||_{(0,\infty)} \\
          &+&||(-\Delta)^{-\frac{1}{2}}\partial_z^2\rho||_{(0,\infty)}+||\nabla\rho||_{(0,\infty)}\,.       
         \end{array}  
         \end{equation}       
                  
   \item Let $u^z, v^z$ satisfy the problem (\ref{FraFor}) and assume 
         $v^z$ is  horizontally band-limited, i.e 
          $$\F v^z(k',z,t)=0 \mbox{ unless } 1\leq R|k'|\leq 4\,.$$
          Then,   
         \begin{equation}\label{C}
           \begin{array}{rclc}
           &&||\partial_t u^z||_{(0,\infty)}+||\nabla^2u^z||_{(0,\infty)}+||(-\Delta')^{-\frac{1}{2}}\partial_z(\partial_t-\partial_z^2)u^z||_{(0,\infty)}\\
          &\lesssim & ||\nabla v^z||_{(0,\infty)}+||(-\Delta')^{-\frac{1}{2}}(\partial_t-\partial_z^2)v^z||_{(0,\infty)}\,.\\    
           \end{array}              
          \end{equation}
   \item Let $v',f'$, satisfy the problem (\ref{Heat2}) and assume 
          $f'$ is horizontally band-limited, i.e 
          $$\F f(k',z,t)=0 \mbox{ unless } 1\leq R|k'|\leq 4\,.$$          
          Then,
          \begin{equation}\label{D}
           ||\nabla'\nabla v'||_{(0,\infty)}+||(\partial_t-\partial_z^2)v'||_{(0,\infty)}\lesssim ||f'||_{(0,\infty)}\,. 
          \end{equation}
%
 \end{enumerate}

\end{proposition}

\subsection{Proof of Proposition \ref{pr1}}

By an easy application of Proposition \ref{prop3}, we will now prove the maximal regularity estimate on the upper  half space.

\begin{proof}[Proof of Proposition \ref{pr1}]\ \\
From Proposition \ref{prop3} we have the following bound for the  vertical component of the velocity $u$

          \begin{eqnarray*} 
           &&||\partial_t u^z||_{(0,\infty)}+||\nabla^2u^z||_{(0,\infty)}+||(-\Delta')^{-\frac{1}{2}}\partial_z(\partial_t-\partial_z^2)u^z||_{(0,\infty)}\\
          &\stackrel{(\ref{C})}{\lesssim} &||\nabla v^z||_{(0,\infty)}+||(-\Delta')^{-\frac{1}{2}}(\partial_t-\partial_z^2)v^z||_{(0,\infty)}\\
          &\stackrel{(\ref{B})}{\lesssim} &||f||_{(0,\infty)}+||\phi||_{(0,\infty)}+||(-\Delta')^{-\frac{1}{2}}\partial_t\rho||_{(0,\infty)}
          +||(-\Delta)^{-\frac{1}{2}}\partial_z^2\rho||_{(0,\infty)}+||\nabla\rho||_{(0,\infty)} \\
           &\stackrel{(\ref{A})}{\lesssim} & ||f||_{(0,\infty)}+||(-\Delta')^{\frac{1}{2}}\partial_t\rho||_{(0,\infty)}+||(-\Delta)^{-\frac{1}{2}}\partial_z^2\rho||_{(0,\infty)}+||\nabla\rho||_{(0,\infty)} \,.       
         \end{eqnarray*}
 Instead for the horizontal components of the velocity $u'$ we have 
        \begin{eqnarray*}
        &&||(\partial_t-\partial_z^2)u'||_{(0,\infty)}+||\nabla'\nabla u'||_{(0,\infty)}\\
        &\stackrel{(\ref{horvel})}{\lesssim}& ||(\partial_t-\partial_z^2)v'||_{(0,\infty)}+||\nabla'\nabla v'||_{(0,\infty)}\\
        &+&||(-\Delta')^{-\frac{1}{2}}(\partial_t-\partial_z^2)\rho||_{(0,\infty)}+||\nabla\rho||_{(0,\infty)}\\
        &+&||(-\Delta')^{-\frac{1}{2}}\partial_z(\partial_t-\partial_z^2)u^z||_{(0,\infty)}+||\partial_z\nabla u^z||_{(0,\infty)}\\
        &\stackrel{(\ref{B}),(\ref{C}),(\ref{D})}{\lesssim}& ||f||_{(0,\infty)}+||(-\Delta')^{-\frac{1}{2}}\partial_t\rho||_{(0,\infty)}+||(-\Delta)^{-\frac{1}{2}}\partial_z^2\rho||_{(0,\infty)}+||\nabla\rho||_{(0,\infty)}\,.        
      \end{eqnarray*}  
 Summing up we obtain
\begin{equation}\label{ESTI}
\begin{array}{rclc}
&&||\partial_t u^z||_{(0,\infty)}+||\nabla^2u^z||_{(0,\infty)}+||(\partial_t-\partial_z^2)u'||_{(0,\infty)}+||\nabla'\nabla u'||_{(0,\infty)}\\
&\lesssim&||f||_{(0,\infty)}+||(-\Delta')^{-\frac{1}{2}}\partial_t\rho||_{(0,\infty)}+||(-\Delta)^{-\frac{1}{2}}\partial_z^2\rho||_{(0,\infty)}+||\nabla\rho||_{(0,\infty)}\,.
\end{array}
\end{equation}
The bound for the $\nabla p$ follows by equations (\ref{STOKES-HALF}) and applying (\ref{ESTI}).
\end{proof}

\subsection{Proof of Proposition \ref{prop3} }
This section is devoted to the proof of Proposition \ref{prop3}, which rely on a series of Lemmas
(Lemma \ref{lemma1}, Lemma \ref{lemma2} and Lemma \ref{lemma3}) that we state here and prove in Section \ref{tec}.

The following Lemmas contain the basic maximal regularity 
estimates for the three auxiliary problems. These estimates,
together with the bandedness assumption in the form  of (\ref{P}), (\ref{Q})
and (\ref{R}) will be the main ingredients for the proof
of Proposition \ref{prop3}.

\begin{lemma}\label{lemma1}\ \\
  Let $u,f$ satisfy  the problem

 \begin{equation}\label{I}
 \left\{\begin{array}{rclc}
        (\partial_z-(-\Delta')^{\frac{1}{2}})u&=&f  \qquad & {\rm for } \quad z>0\,,\\
        u&\rightarrow&0   \qquad & {\rm for } \quad z\rightarrow\infty\,
      \end{array}\right.
  \end{equation}   
    and assume $f$ to be horizontally band-limited, i.e
   $$\F f(k',z,t)=0 \quad \mbox{ unless }\quad 1\leq R|k'|\leq 4\,.$$
  Then, 
   \begin{equation}\label{1}
    ||\nabla u||_{(0,\infty)}\lesssim||f||_{(0,\infty)}\,.
   \end{equation}
\end{lemma}

\begin{lemma}\label{lemma2}\ \\ 	
   Let $u,f,g=g(x',t)$ satisfy  the problem
    
  \begin{equation}\label{II}
 \left\{\begin{array}{rclc}
        (\partial_z+(-\Delta')^{\frac{1}{2}})u&=&f   \qquad & {\rm for } \quad z>0\,,\\
        u&=&g   \qquad & {\rm for } \quad z=0\,
      \end{array}\right.
  \end{equation}
 and define the constant extension $\tilde g(x',z,t):=g(x',t).$ 
 Assume $f$ and $g$ to be horizontally band-limited, i.e
  $$\F f(k',z,t)=0 \quad\mbox{ unless }\quad 1\leq R|k'|\leq 4\,$$
  and
  $$\F g(k',z,t)=0 \quad\mbox{ unless }\quad 1\leq R|k'|\leq 4\,.$$
  Then
  \begin{equation}\label{3}
   ||\nabla u||_{(0,\infty)}\lesssim||f||_{(0,\infty)}+||\nabla'\tilde{g}||_{(0,\infty)}\,.
     \end{equation}
\end{lemma}	
  
\begin{remark}
 Clearly if $g=0$ in Lemma \ref{lemma2}, then we have
  \begin{equation}\label{2}
  ||\nabla u||_{(0,\infty)}\lesssim||f||_{(0,\infty)}\,.
  \end{equation}
\end{remark}
		  
\begin{lemma}\label{lemma3}\ \\
   Let $u,f$ satisfy  the problem
     \begin{equation}\label{III}
     \left\{\begin{array}{rclc}
	(\partial_t-\Delta) u&=&f \qquad & {\rm for } \quad z>0\,,\\ 
	u&=&0 \qquad & {\rm for } \quad z=0\,,\\
	u&=&0 \qquad & {\rm for } \quad t=0\,\\
      \end{array}\right.
     \end{equation}
     and assume $f$ to be horizontally band-limited, i.e
   $$\F f(k',z,t)=0 \quad\mbox{ unless }\quad 1\leq R|k'|\leq 4\,.$$

Then, 				
 \begin{equation}\label{4}
 ||(\partial_t-\partial_z^2)u||_{(0,\infty)}+||\nabla'\nabla u||_{(0,\infty)}\lesssim ||f||_{(0,\infty)}\,.
 \end{equation}
       
 \end{lemma}       
               
\vspace{2cm}

\begin{proof} [Proof of Proposition \ref{prop3}]\ \\
 
\begin{enumerate}

 \item 

Subtracting the quantity $(\partial_z-(-\Delta')^{\frac{1}{2}})(f^z+\partial_z\rho)$
 from both sides of equation (\ref{FraBack})
 and then multiplying the new equation by $(-\Delta)^{-\frac{1}{2}}$ we get

\begin{eqnarray*}
&&(\partial_z-(-\Delta')^{\frac{1}{2}})(-\Delta')^{-\frac{1}{2}}(\phi-f^z-\partial_z\rho)\\
&=&\nabla'\cdot(-\Delta')^{-\frac{1}{2}} f'+f^z-(-\Delta')^{-\frac{1}{2}}\partial_t\rho+\partial_z \rho-(-\Delta')^{\frac{1}{2}}\rho\,.
\end{eqnarray*}

From the basic estimate (\ref{1}) we obtain

\begin{eqnarray*}
&&||\nabla'(-\Delta')^{-\frac{1}{2}}(\phi-f^z-\partial_z\rho)||_{(0,\infty)}\lesssim ||\nabla'\cdot(-\Delta')^{-\frac{1}{2}} f'||_{(0,\infty)}\\
&+&||f^z||_{(0,\infty)}+||(-\Delta')^{-\frac{1}{2}}\partial_t\rho||_{(0,\infty)}+||\partial_z \rho||_{(0,\infty)}+||(-\Delta')^{\frac{1}{2}}\rho||_{(0,\infty)}\,.
\end{eqnarray*}

Thanks to the bandedness assumption in the form of  (\ref{P})  and (\ref{Q}) we have
\begin{eqnarray*}
&&||\phi-f^z-\partial_z\rho||_{(0,\infty)}\\
&\lesssim& || f'||_{(0,\infty)}+||f^z||_{(0,\infty)}+||(-\Delta')^{-\frac{1}{2}}\partial_t\rho||_{(0,\infty)}+||\partial_z \rho||_{(0,\infty)}+||\nabla'\rho||_{(0,\infty)}
\end{eqnarray*}
and from this we obtain easily the desired estimate (\ref{A}).
\item 

After multiplying the equation  (\ref{Heat1}) by $(-\Delta')^{-\frac{1}{2}}$, 
the application of (\ref{4}) to $(-\Delta')^{-\frac{1}{2}}v^z$ yields

\begin{eqnarray*}
&&||(-\Delta')^{-\frac{1}{2}}(\partial_t-\partial_z^2)v^z||_{(0,\infty)}+||(-\Delta')^{-\frac{1}{2}}\nabla'\nabla v^z||_{(0,\infty)}\\
&\lesssim& ||f^z||_{(0,\infty)}+||\phi||_{(0,\infty)}+||\nabla'\cdot(-\Delta')^{-\frac{1}{2}}f'||_{(0,\infty)}\\
&+&||(-\Delta')^{-\frac{1}{2}}(\partial_t-\partial_z^2)\rho||_{(0,\infty)}+||(-\Delta')^{\frac{1}{2}}\rho||_{(0,\infty)}\,.
\end{eqnarray*}

The estimate (\ref{B}) follows after observing (\ref{Q}) and applying the triangle inequality to the second to last
 term on the right hand side.
 
\item

We need to estimate the the three terms on the right hand side of (\ref{C}) separately.
We start with the term $\nabla^2 u^z$:
 since $||\nabla^2 u^z||_{(0,\infty)}\leq||\nabla'\nabla u^z||_{(0,\infty)}+||\partial_z^2 u^z||_{(0,\infty)}$, 
we tackle the term $\nabla'\nabla u^z$ and $\partial_z^2 u^z$ separately.
First multiply by $\nabla'$ the equation (\ref{FraFor}).  An application
 of the estimate (\ref{2}) to $\nabla' u^z$ yields
\begin{equation}\label{G}||\nabla\nabla'u^z||_{(0,\infty)}\lesssim ||\nabla' v^z||_{(0,\infty)} .\end{equation}
Now multiplying the equation (\ref{FraFor}) by $\partial_z^2$ 
\begin{equation}\label{E}\partial_z^2 u^z=-(-\Delta')^{\frac{1}{2}}\partial_z u^z+\partial_z v^z=-\Delta'u^z-(-\Delta')^{\frac{1}{2}}v^z+\partial_z v^z\end{equation}
and using the bandedness assumption in the form (\ref{Q}) we have 
\begin{equation}\label{i}
\begin{array}{rclc}
||\partial_z^2 u^z||_{(0,\infty)}&\leq& ||\nabla'^2 u^z||_{(0,\infty)}+||\nabla v^z||_{(0,\infty)}\\
&\stackrel{(\ref{G})}{\leq}&||\nabla v^z||_{(0,\infty)}\,.
\end{array}
\end{equation}

 The second term  of (\ref{C}), i.e $(-\Delta')^{-\frac{1}{2}}\partial_z(\partial_t-\partial_z^2)u^z$, can be bounded in the following way:
We multiply the equation  (\ref{FraFor}) by $(-\Delta')^{-\frac{1}{2}}(\partial_t-\partial_z^2)$

\begin{equation*}
    \left\{\begin{array}{rclc}
(\partial_z+(-\Delta')^{\frac{1}{2}})(-\Delta')^{-\frac{1}{2}}(\partial_t-\partial_z^2)u^z&=& (-\Delta')^{-\frac{1}{2}}(\partial_t-\partial_z^2)v^z \qquad & {\rm for } \quad   z>0,\\
(-\Delta')^{-\frac{1}{2}}(\partial_t-\partial_z^2)u^z&=&(-\Delta')^{-\frac{1}{2}}\partial_zv^z \qquad & {\rm for } \quad   z=0,
    \end{array}\right.
  \end{equation*}

where we have used  that at $z=0$

$$(\partial_t-\partial_z^2)u^z=-\partial_z^2 u^z\stackrel{(\ref{E})}{=}\partial_zv^z.$$
Applying (\ref{3}) to $(-\Delta')^{-\frac{1}{2}}(\partial_t-\partial_z^2)u^z$ and using the bandedness assumption in the form of (\ref{P}),

\begin{equation}\label{F}
||\nabla(-\Delta')^{-\frac{1}{2}}(\partial_t-\partial_z^2)u^z||_{(0,\infty)}\lesssim ||(-\Delta')^{-\frac{1}{2}}(\partial_t-\partial_z^2)v^z||_{(0,\infty)}+||\partial_zv^z||_{(0,\infty)}\,.
\end{equation}
%
%
%

Finally we can bound the last term of (\ref{C}), i.e $\partial_t u^z$:
We observe that $\partial_t u^z=(\partial_t -\partial_z^2 )u^z+\partial_z^2 u^z$ thus 

\begin{equation}\label{H}
||\partial_t u^z||_{(0,\infty)}\leq ||(\partial_t -\partial_z^2 )u^z||_{(0,\infty)}+||\partial_z^2 u^z||_{(0,\infty)}\,.
\end{equation}

For the first term in the right hand side of (\ref{H})
we notice that 
\begin{eqnarray*}
||(\partial_t-\partial_z^2)u^z||_{(0,\infty)}
&\stackrel{(\ref{P})}{\leq}&||(-\Delta')^{-\frac{1}{2}}\nabla'(\partial_t-\partial_z^2)u^z||_{(0,\infty)}\\
&\stackrel{(\ref{F})}{\lesssim}&||(-\Delta')^{-\frac{1}{2}}(\partial_t-\partial_z^2)v^z||_{(0,\infty)}+||\partial_z v^z||_{(0,\infty)}\\
&\lesssim&||(-\Delta')^{-\frac{1}{2}}(\partial_t-\partial_z^2)v^z||_{(0,\infty)}+||\nabla v^z||_{(0,\infty)}\,.
\end{eqnarray*}
The  second term  on the right hand side of (\ref{H}) is bounded in (\ref{i}). 
Thus we have the following bound for $\partial_t u$
\begin{equation}\label{M}||\partial_t u^z||_{(0,\infty)}\leq ||(-\Delta')^{-\frac{1}{2}}(\partial_t-\partial_z^2)v^z||_{(0,\infty)}+||\nabla v^z||_{(0,\infty)}\,.\end{equation}

Putting together all the above we obtain the desired estimate.
%
 
\item From the defining equation (\ref{Heat2}), the  basic estimate (\ref{4})
 and the bandedness assumption in form of (\ref{R}), we get

$$||(\partial_t-\partial_z^2)v'||_{(0,\infty)}+||\nabla'\nabla v'||_{(0,\infty)}\lesssim ||f'||_{(0,\infty)}\,.$$


\end{enumerate}
\end{proof}

\subsection{Proof of Theorem \ref{th1}}\label{ProofTh1}

Let $u,p,f$ be the solutions of the non-stationary Stokes equations in the strip $0<z<1$ (\ref{STOKES-STRIP}).
Then $\tilde u=\eta u,\tilde p=\eta p$ (with $\eta$ defined in (\ref{cutoff}) satisfy (\ref{UHS}), namely
 \begin{equation*}
  \left\{\begin{array}{rclc}
      \partial_t \tilde u-\Delta \tilde u+\nabla \tilde p &=& \tilde f \qquad & {\rm for } \quad z>0\,,\\
        \nabla\cdot \tilde u &=& \tilde\rho \qquad & {\rm for } \quad z>0 \,,\\
        \tilde u &=&  0 \qquad & {\rm for } \quad z=0\,,\\
         \tilde u &=& 0  \qquad  & {\rm for } \quad t=0 \,,\\
         \end{array}\right.        
 \end{equation*}
where 
\begin{equation}\label{Defi-bis}
\tilde f:=\eta f-2(\partial_z \eta)\partial_z u-(\partial_z^2\eta )u+(\partial_z\eta )pe_z, \qquad \qquad  \tilde\rho:=(\partial_z\eta )u^z\,.
\end{equation}
Since, by assumption $f,\rho$  are horizontally band-limited , 
then also $\tilde f$ and $\tilde \rho$  satisfy the horizontal bandedness assumption (\ref{BC1}) and (\ref{BC2}) respectively.
We can therefore apply Proposition \ref{pr1} to the upper half space problem (\ref{UHS}) and get
\begin{eqnarray*}
&&||(\partial_t -\partial_z^2)\tilde{u}'||_{(0,\infty)}+ ||\nabla'\nabla \tilde{u}'||_{(0,\infty)}+||\partial_t \tilde{u}^z||_{(0,\infty)}+||\nabla^2 \tilde{u}^z||_{(0,\infty)}+||\nabla \tilde{p}||_{(0,\infty)}\\
&\lesssim&||\tilde{f}||_{(0,\infty)}+||(-\Delta')^{-\frac{1}{2}}\partial_t \tilde{\rho}||_{(0,\infty)}+||(-\Delta')^{-\frac{1}{2}}\partial_z^2 \tilde{\rho} ||_{(0,\infty)}+||\nabla \tilde{\rho}||_{(0,\infty)}\,.
\end{eqnarray*}
By symmetry, we also have the same maximal regularity estimates in the lower half space. Indeed, let
 $\tilde{\tilde u} ,\tilde{\tilde p}$ satisfy the equation 
 \begin{equation}\label{LHS}
  \left\{\begin{array}{rclc}
      \partial_t \tilde{\tilde u}-\Delta \tilde{\tilde u}+\nabla \tilde {\tilde p} &=& \tilde{\tilde{f}} \qquad & {\rm for } \quad z<1\,,\\
        \nabla\cdot \tilde {\tilde u} &=& \tilde{\tilde\rho} \qquad & {\rm for } \quad z<1 \,,\\
        \tilde {\tilde u} &=&  0 \qquad & {\rm for } \quad z=1\,,\\
         \tilde {\tilde u} &=& 0  \qquad  & {\rm for } \quad t=0 \,,\\
         \end{array}\right.        
 \end{equation}
where 
\begin{equation}\label{Defi2}
\tilde{\tilde f}:=(1-\eta) f-2(\partial_z (1-\eta))\partial_z u-(\partial_z^2(1-\eta) )u+(\partial_z(1-\eta) )pe_z, \qquad   \tilde{\tilde\rho}:=(\partial_z(1-\eta) )u^z\,.
\end{equation}
Again by Proposition \ref{prop3} we have 
\begin{eqnarray*}
&&||(\partial_t -\partial_z^2)\tilde{\tilde{u}}'||_{(-\infty,1)}+ ||\nabla'\nabla \tilde{\tilde{u}}'||_{(-\infty,1)}+ ||\partial_t \tilde{\tilde{u}}^z||_{(-\infty,1)}+||\nabla^2 \tilde{\tilde{u}}^z||_{(-\infty,1)}+||\nabla \tilde{\tilde p}||_{(-\infty,1)}\\
&\lesssim&||\tilde{\tilde{f}}||_{(-\infty,1)}+||(-\Delta')^{-\frac{1}{2}}\partial_t \tilde{\tilde{\rho}}||_{(-\infty,1)}+||(-\Delta')^{-\frac{1}{2}}\partial_z^2 \tilde{\tilde{\rho}} ||_{(-\infty,1)}+||\nabla \tilde{\tilde{\rho}}||_{(-\infty,1)},
\end{eqnarray*}
where $||\cdot||_{(-\infty,1)}$ is the analogue of (\ref{NORM-HALF}) (see Section (\ref{notations}) for notations).
Since $u=\tilde u+\tilde{\tilde u}$ in the strip $[0,L)^{d-1}\times (0,1)$, by the triangle inequality and using the maximal regularity estimates above, we get
\begin{eqnarray*}
&&||(\partial_t -\partial_z^2)u'||_{(0,1)}+||\nabla'\nabla u'||_{(0,1)}+||\partial_t u^z||_{(0,1)}+||\nabla^2 u^z||_{(0,1)}+||\nabla p||_{(0,1)}\\
&\lesssim&||(\partial_t -\partial_z^2)\tilde{u}'||_{(0,\infty)}+||(\partial_t -\partial_z^2)\tilde{\tilde{u}}'||_{(-\infty,1)}+||\nabla'\nabla \tilde{u}'||_{(0,\infty)}+||\nabla'\nabla \tilde{\tilde{u}}'||_{(-\infty,1)}\\
&+&||\partial_t \tilde{u}^z||_{(0,\infty)}+ ||\partial_t \tilde{\tilde{u}}^z||_{(-\infty,1)}+ ||\nabla^2 \tilde{u}^z||_{(0,\infty)}+ ||\nabla^2 \tilde{\tilde{u}}^z||_{(-\infty,1)}\\
&+&||\nabla \tilde{p}||_{(0,\infty)}+||\nabla \tilde{\tilde p}||_{(-\infty,1)}\\
&\lesssim&||\tilde{f}||_{(0,\infty)}+||\tilde{\tilde{f}}||_{(-\infty,1)}+||(-\Delta')^{-\frac{1}{2}}\partial_t \tilde{\rho}||_{(0,\infty)}+||(-\Delta')^{-\frac{1}{2}}\partial_t \tilde{\tilde{\rho}}||_{(-\infty,1)}\\
&+&||(-\Delta')^{-\frac{1}{2}}\partial_z^2 \tilde{\rho} ||_{(0,\infty)}+||(-\Delta')^{-\frac{1}{2}}\partial_z^2 \tilde{\tilde{\rho}} ||_{(-\infty,1)}+||\nabla \tilde{\rho}||_{(0,\infty)}+||\nabla \tilde{\tilde{\rho}}||_{(-\infty,1)}\,.
\end{eqnarray*}
By the definitions of $\tilde{f}$ and $\tilde{\tilde{f}}$ we get

$$||\tilde{f}||_{(0,\infty)}+||\tilde{\tilde{f}}||_{(-\infty,1)}\lesssim ||f||_{(0,1)}+||\partial_z u||_{(0,1)}+||u||_{(0,1)}+||p||_{(0,1)}$$
and similarly for $\tilde{\rho}$ and $\tilde{\tilde{\rho}}$ we have
$$||\nabla \tilde{\rho}||_{(0,\infty)}+||\nabla \tilde{\tilde{\rho}}||_{(-\infty,1)}\lesssim ||\nabla u||_{(0,1)}+||u||_{(0,1)}$$
$$ ||(-\Delta')^{-\frac{1}{2}}\partial_t \tilde{\rho}||_{(0,\infty)}+||(-\Delta')^{-\frac{1}{2}}\partial_t \tilde{\tilde{\rho}}||_{(-\infty,1)}\lesssim || (-\Delta')^{-\frac{1}{2}}\partial_t u||_{(0,1)} $$
and 
\begin{eqnarray*}
&&||(-\Delta')^{-\frac{1}{2}}\partial_z^2 \tilde{\rho} ||_{(0,\infty)}+||(-\Delta')^{-\frac{1}{2}}\partial_z^2 \tilde{\tilde{\rho}} ||_{(-\infty,1)} \\
&\lesssim&||(-\Delta')^{-\frac{1}{2}}u^z||_{(0,1)}+||(-\Delta')^{-\frac{1}{2}}\partial_zu^z||_{(0,1)}+||(-\Delta')^{-\frac{1}{2}}\partial^2_zu^z||_{(0,1)}\,.
\end{eqnarray*}
Therefore, collecting the estimates, we have 
\begin{eqnarray*}
&&||(\partial_t -\partial_z^2)u'||_{(0,1)}+||\nabla'\nabla u'||_{(0,1)}+||\partial_t u^z||_{(0,1)}+||\nabla^2 u^z||_{(0,1)}+||\nabla p||_{(0,1)}\\
&\lesssim&||f||_{(0,1)}+||p||_{(0,1)}+||\nabla u||_{(0,1)}+||u||_{(0,1)}\\
&+& || (-\Delta')^{-\frac{1}{2}}\partial_t u||_{(0,1)}+||(-\Delta')^{-\frac{1}{2}}u^z||_{(0,1)}+||(-\Delta')^{-\frac{1}{2}}\partial_zu^z||_{(0,1)}+||(-\Delta')^{-\frac{1}{2}}\partial^2_zu^z||_{(0,1)}\,.
\end{eqnarray*}
Incorporating the horizontal bandedness assumption we find
 \begin{eqnarray*}
 ||\partial_z u||_{(0,1)}&\leq& R  ||\nabla'\partial_z u||_{(0,1)}\,,\\
 ||u||_{(0,1)}&\leq& R^2||(\nabla')^2 u||_{(0,1)} \,, \\    
 ||p||_{(0,1)}&\leq& R||\nabla'p||_{(0,1)} ,\\
 ||\nabla u||_{(0,1)}&\leq& R||\nabla'\nabla u||_{(0,1)},\\
 || (-\Delta')^{-\frac{1}{2}}\partial_t u||_{(0,1)}&\leq& R || \partial_t u||_{(0,1)}\,,\\
 ||(-\Delta')^{-\frac{1}{2}}u^z||_{(0,1)}&\leq&  R^3||\nabla'^2u^z||_{(0,1)}\,,\\
 ||(-\Delta')^{-\frac{1}{2}}\partial_zu^z||_{(0,1)}&\leq& R^2 ||\nabla'\partial_zu^z||_{(0,1)}\,,\\
 ||(-\Delta')^{-\frac{1}{2}}\partial^2_zu^z||_{(0,1)}&\leq& R||\partial^2_zu^z||_{(0,1)} \,.      
 \end{eqnarray*}               
Thus, for $R<R_0$ where $R_0$ is sufficiently small, all the terms in the right hand side,
except $f$ can be absorbed into the left hand side and the conclusion follows.

\section{Proof of main technical lemmas}\label{tec}

\begin{remark}\label{TIME}
 In the proof of Lemma \ref{lemma1}, Lemma \ref{lemma2} and Lemma \ref{lemma3} we  will derive inequalities between quantities  where $t$ is integrated between $0$ and $\infty$. 
  From the proof it is clear that the same inequalities are true  with $t$ integrated between $0$ and $t_0$ with constants that
  are not depending on $t_0$. Therefore
  dividing by $t_0$ and taking $\limsup_{t_{0}\rightarrow \infty}$ (see (\ref{LTaHA})) we shall obtain the desired estimates in terms of the interpolation norm (\ref{NORM-HALF}).
\end{remark}

\subsection{Proof of Lemma \ref{lemma1}}

\begin{proof}[Proof of Lemma \ref{lemma1}]\ \\
	      In order to simplify the notations, in what follows we will omit
	      the dependency of the functions from the time variable.	      
              It is enough to show 
	      \begin{equation*}
	       \label{U}||\nabla' u||_{(0,\infty)}\lesssim ||f||_{(0,\infty)},
	      \end{equation*}
              since, by equation (\ref{I}) $\partial_z u=(-\Delta')^{\frac{1}{2}} u+f$.
              We claim that, in order to prove (\ref{U}), it is enough to show
              \begin{equation}\label{V} 
	       \sup_z\langle|\nabla' u|\rangle'\lesssim \sup_z \langle|f|\rangle'
	      \end{equation}
              and 
              \begin{equation}\label{Z} 
	       ||\nabla'u||_{(0,\infty)}\lesssim \int\langle|f|\rangle'\frac{dz}{z}\,.
	      \end{equation}
               Indeed, by definition of the norm $||\cdot||_{(0,\infty)}$ (see (\ref{NORM-HALF}))
               if we select an arbitrary decomposition $\nabla'u=\nabla'u_1+\nabla'u_2$, where $u_1$ and $u_2$ are 
               solutions of the problem (\ref{I}) with right hand sides $f_1$ and $f_2$ respectively, we have 
              
              \begin{eqnarray*}
               ||\nabla' u||_{(0,\infty)}&\leq& ||\nabla' u_1||_{(0,\infty)}+\sup_{z}\langle|\nabla' u_2|\rangle'\\
               &\leq& \int\langle| f_1|\rangle'\frac{dz}{z}+\sup_z \langle|f_2|\rangle'\,.\\
              \end{eqnarray*}                           
               Passing to the infimum over all the decompositions of $f$ we obtain
               $$||\nabla' u||_{(0,\infty)}\lesssim ||f||_{(0,\infty)}.$$
              We recall that by Duhamel's principle we have the following representation
              \begin{equation}\label{DU}
              u(x',z)=\int_z^{\infty} u_{x',z_0}(z)dz_0,
              \end{equation}
              where $u_{z_0}$  is the harmonic extension of $f(\cdot,z_0)$ onto $\{z<z_0\}$, i.e it solves the boundary value problem 

              \begin{equation}\label{FraBackDu}
                \left\{\begin{array}{rclc}
                      (\partial_z-(-\Delta')^{\frac{1}{2}}) u_{z_0}&=&0  \qquad & {\rm for } \quad z<z_0\,,\\
                       u_{z_0}&=&f  \qquad & {\rm for } \quad z=z_0\,.\\
                       \end{array}\right.        
               \end{equation}                
                     \underline{Argument for (\ref{V}):} 
                     \newline
                     Using the representation of the solution of (\ref{FraBackDu}) via the Poisson kernel, i.e
                     $$u_{z_0}(x',z)=\int\frac{z_0-z}{(|x'-y'|^2+(z_0-z)^2)^{\frac{d}{2}}}f(x',z_0) dy'$$
                     we obtain the following bounds
                     \begin{equation}\label{AA}
                     \langle|\nabla' u_{z_0}(\cdot,z)|\rangle'\lesssim 
                     \left\{\begin{array}{lll}&&\langle|\nabla' f(\cdot,z_0)|\rangle',\\                    
                      &\frac{1}{(z_0-z)}&\langle| f(\cdot,z_0)|\rangle',\\                                       
                      &\frac{1}{(z_0-z)^2}&\langle|\nabla'(-\Delta')^{-1}f(\cdot,z_0)|\rangle'.\\                                       
                      \end{array}\right.                                       
                      \end{equation}                                      
		     By using the bandedness assumption in the form of (\ref{BAND1}) and (\ref{BAND2}), we have   
		     $$\langle|\nabla' u_{z_0}(\cdot,z)|\rangle'\lesssim \min \left\{\frac{1}{R}, \frac{R}{(z_0-z)^2}\right\}\langle|f(\cdot,z_0)|\rangle',$$
		     hence		     
                     \begin{eqnarray*}
			\langle|\nabla' u(\cdot,z)|\rangle'&\lesssim& \int_z^{\infty}\min \left\{\frac{1}{R},\frac{R}{(z_0-z)^2}\right\}\langle|f(\cdot,z_0)|\rangle'dz_0\\
			&\lesssim&\sup_{z_0\in (0,\infty)}\langle|f(\cdot,z_0)|\rangle'\int_z^{\infty}\min \left\{\frac{1}{R},\frac{R}{(z_0-z)^2}\right\}dz_0\\
			&\lesssim&\sup_{z_0\in (0,\infty)}\langle|f(\cdot,z_0)|\rangle',
	             \end{eqnarray*}	            
                     which, passing to the supremum in $z$, implies (\ref{V}).
                     \newline
                    From the above and applying Fubini's rule, we also have
                     \begin{align*}                       
			\int_{0}^{\infty}\langle|\nabla' u(\cdot,z)|\rangle' dz&\leq \int_0^{\infty} \int_z^{\infty}\min \left\{\frac{1}{R},\frac{R}{(z_0-z)^2}\right\}\langle|f(\cdot,z_0)|\rangle'dz_0 dz\numberthis\label{UNW1}\\
			&\leq\int_0^{\infty} \int_0^{z_0}\min \left\{\frac{1}{R},\frac{R}{(z_0-z)^2}\right\}dz\langle|f(\cdot,z_0)|\rangle'dz_0 \nonumber\\
			&\lesssim\int_0^{\infty}\langle|f(\cdot,z)|\rangle' dz \nonumber \,.\\
                     \end{align*}                     
		     \underline{Argument for (\ref{Z}):}
		     \newline
                      Let us consider $\chi_{2H\leq z\leq 4H}f$ and let $u_H$ be the solution to
                      $$(\partial_z-(-\Delta')^{\frac{1}{2}})u_{H}=\chi_{2H\leq z\leq 4H}f.$$
		    We claim
                     \begin{equation}\label{X}
			\sup_{z\leq H}\langle|\nabla' u_{H}|\rangle'\leq \int_0^{\infty}\langle|\chi_{2H\leq z\leq 4H}f|\rangle'\frac{dz}{z}
                     \end{equation}                    
		     and		     
		     \begin{equation}\label{XX}
			\int_H^{\infty}\langle|\nabla' u_{H}|\rangle'\frac{dz}{z}\leq \int_0^{\infty}\langle|\chi_{2H\leq z\leq 4H}f|\rangle'\frac{dz}{z}\,.
                     \end{equation}                          
                      From estimate (\ref{X}) and (\ref{XX}) the statement (\ref{Z}) easily follow.
                      Indeed, choosing $H=2^{n-1}$ and summing up over the dyadic intervals, we have
                      \begin{eqnarray*}
                      ||\nabla'u||&\leq& \sum_{n\in \Z}||\nabla'u_{2^{n-1}}||_{(0,\infty)}\\
                      &\leq&\sup_{z\leq 2^{n-1}}\langle|\nabla' u_{2^{n-1}}|\rangle'+\int_{2^{n-1}}^{\infty}\langle|\nabla' u_{2^{n-1}}|\rangle'\frac{dz}{z}\\
                      &\leq&\sum_{n\in \Z} \int_0^{\infty}\langle| \chi_{2^n\leq z\leq 2^{n+1}}f|\rangle'\frac{dz}{z}\\
                      &=&\int_0^{\infty}\langle|f|\rangle'\frac{dz}{z}\,.
                      \end{eqnarray*}                         
                        Argument for (\ref{X}):  
                                Fix $z\leq H$. Then, we have
                                  \begin{eqnarray*}
				      \langle|\nabla'u_H|\rangle'&\stackrel{(\ref{AA})}{\leq}&\int_z^{\infty}\frac{1}{(z_0-z)}\langle|\chi_{2H\leq z\leq 4H}f(\cdot,z_0)|\rangle' dz_0\\
 				      &\lesssim&\int_{2H}^{4H}\frac{1}{(z_0-z)}\langle|\chi_{2H\leq z\leq 4H}f(\cdot,z_0)|\rangle' dz_0\\ 
 				      &\lesssim& \frac{1}{H}\int_{2H}^{4H}\langle|\chi_{2H\leq z\leq 4H}f(\cdot,z_0)|\rangle' dz_0\\                                 
 				      &\leq&\int_{2H}^{\infty}\langle|\chi_{2H\leq z\leq 4H}f(\cdot,z_0)|\rangle' \frac{dz_0}{z_0}\\
 				      &\leq&\int_{0}^{\infty}\langle|\chi_{2H\leq z\leq 4H}f(\cdot,z_0)|\rangle' \frac{dz_0}{z_0}\,.\\
                                  \end{eqnarray*}                                
                                  Taking the supremum over all $z$ proves (\ref{X}). 
                                  \newline
                                 Argument for (\ref{XX}):
                                  For $z\geq H$ we have
                                 \begin{eqnarray*}
                                  \int_H^{\infty}\langle|\nabla'u_H|\rangle'\frac{dz}{z}&\lesssim& \frac{1}{H}\int_0^{\infty}\langle|\nabla'u_H|\rangle' dz\\
                                  &\stackrel{(\ref{UNW1})}{\lesssim}&\frac{1}{H}\int_0^{\infty}\langle|\chi_{2H\leq z\leq 4H}f|\rangle' dz\\
                                  &=&\frac{1}{H}\int_{2H}^{4H}\langle|\chi_{2H\leq z\leq 4H}f|\rangle' dz\\
                                  &\lesssim& \int_0^{\infty}\langle|\chi_{2H\leq z\leq 4H}f|\rangle'\frac{dz}{z}\,.\\
                                 \end{eqnarray*}
\end{proof}

\subsection{Proof of Lemma \ref{lemma2}}

\begin{proof}[Proof of Lemma \ref{lemma2}]\ \\		      
                     Let us first assume $g=0$. It is enough to show            	             
	             \begin{equation}\label{SupEst}
	             \sup_z\langle|\nabla'u|\rangle'\lesssim \sup_{z}\langle|f|\rangle'
	             \end{equation}	             
	             and	             
	             \begin{equation}\label{WeiEst}
	             \int_0^{\infty}\langle|\nabla'u|\rangle'\frac{dz}{z}\lesssim \int_{0}^{\infty}\langle|f|\rangle'\frac{dz}{z}\,.
	             \end{equation}	             
	             Recall that by Duhamel's principle we have the following representation
                     \begin{equation}\label{Du}
                      u(z)=\int_0^z u_{z_0}(\cdot,z)dz_0,
                     \end{equation}
                     where $u_{z_0}$  is the harmonic extension of $f(z_0)$ onto $\{z>z_0\}$, i.e it solves the boundary value problem         
                      \begin{equation}\label{FraForDu}
                        \left\{\begin{array}{rclc}
			  (\partial_z+(-\Delta')^{\frac{1}{2}}) u_{z_0}&=&0 \qquad & {\rm for } \quad z>z_0\,,\\
			   u_{z_0}&=&f \qquad & {\rm for } \quad z=z_0\,.\\
		           \end{array}\right.        
                      \end{equation}                   
                    From the Poisson's kernel representation we learn that 
                    $$\langle|\nabla' u_{z_0}(\cdot,z)|\rangle'\lesssim \begin{cases}
                                                                  \langle|\nabla' f(\cdot,z_0)|\rangle'\,,\\
                                                                  \frac{1}{(z-z_0)^2}\langle|\nabla'(-\Delta')^{-1}f(\cdot,z_0)|\rangle'\,.
                                                                  \end{cases}$$
                    Using the bandedness assumption in the form of (\ref{BAND1}) and (\ref{BAND2}) 
                    $$\langle|\nabla' u_{z_0}(\cdot,z)|\rangle'\lesssim \min\{\frac{1}{R},\frac{R}{(z-z_0)^2}\}\langle| f(\cdot,z_0)|\rangle'$$
                    and observing  (\ref{Du}), we obtain 
                    \begin{equation}\label{AUX} 
                    \begin{array}{rclc}     
                    \langle|\nabla' u(\cdot,z)|\rangle'                               
		      &\lesssim& \int_0^{z}\min\left\{\frac{1}{R},\frac{R}{(z-z_0)^2}\right\}\langle| f(\cdot,z_0)|\rangle' dz_0\\
		       &\leq&\sup_{z_0}\langle| f(\cdot,z_0)|\rangle'\int_0^{z}\min\left\{\frac{1}{R},\frac{R}{(z-z_0)^2}\right\} dz_0\\
		       &\lesssim&\sup_{z_0}\langle| f(\cdot,z_0)|\rangle'\,. 
                    \end{array}    
                    \end{equation}                                                       
                    Estimate (\ref{SupEst}) follows from (\ref{AUX}) by passing to the the supremum in $z$.     
                    \newline
                    From the above (\ref{AUX}), multiplying by the weight $\frac{1}{z}$ and observing that $z>z_0$ we have                    
                   \begin{equation}\label{AUX2}
		      \langle|\nabla' u(\cdot,z)|\rangle'\frac{1}{z}\lesssim \int_0^{z}\min\left\{\frac{1}{R},\frac{R}{(z-z_0)^2}\right\}\langle| f(\cdot,z_0)|\rangle'\frac{dz_0}{z_0}\,.
                    \end{equation}                   
                    After integrating in $z\in (0,\infty)$ and applying Young's estimate  we get (\ref{WeiEst}).                                      
		     \vspace{1cm}
		     
                     Let's assume now the general case, with $g\neq 0$. We want to prove (\ref{3}). 
                     Recall that by definition $\tilde g(x',z):=g(x')$ and consider $u-\tilde g$. 
                     By construction it satisfies                                                 
                     \begin{equation*}
                        \left\{\begin{array}{rclc}
			  (\partial_z+(-\Delta')^{-\frac{1}{2}})(u-\tilde g)&=&f-(-\Delta')^{-\frac{1}{2}}g  \qquad & {\rm for } \quad z>0\,,\\
			  u-\tilde g&=&0   \qquad & {\rm for } \quad z=0 \,.\\ 
			     \end{array}\right.        
                      \end{equation*}                 
                    Using the first part of the proof of (\ref{2}) and triangle inequality, we have                  
                    $$||\nabla u||_{(0,\infty)}\lesssim ||\nabla \tilde g||_{(0,\infty)}+||f||_{(0,\infty)}+||(-\Delta')^{\frac{1}{2}}\tilde g||_{(0,\infty)}\,.$$                    
                    Therefore by the bandedness assumption in the form of (\ref{Q}) we can conclude (\ref{3}).                   
  \end{proof}   
                 
 \subsection{Proof of Lemma \ref{lemma3}}

\begin{proof}[Proof of Lemma \ref{lemma3}]\ \\         
		  We will show that, for the non-homogeneous heat 
		  equation with Dirichlet boundary condition   
		    \begin{equation}\label{A1}
                        \left\{\begin{array}{rclc}
			(\partial_t-\Delta)u&=&f  \qquad & {\rm for } \quad   z>0\,,\\
			u&=&0  \qquad & {\rm for } \quad    z=0\,,\\
			u&=&0  \qquad & {\rm for } \quad    t=0\,,\\
		     	 \end{array}\right.        
                      \end{equation}   
                      we have the following estimates
                    \begin{equation}\label{A1.1}
		     \int\left(\langle|(\partial_t-\partial_z^2)u(\cdot,z,\cdot)|\rangle+\langle|\nabla'^2 u(\cdot,z,\cdot)|\rangle\right)\frac{dz}{z}\lesssim \int\langle|f(\cdot,z,\cdot)|\rangle \frac{dz}{z}\,,
		    \end{equation}             
		   \begin{equation}\label{A1.2}
		    \langle|\nabla'\partial_z u(\cdot,z,\cdot)|_{z=0}\rangle\lesssim \int\langle|f(\cdot,z,\cdot)|\rangle \frac{dz}{z}\,,
		   \end{equation} 
		   \begin{equation}\label{A1.3}
		    \sup_{z}\langle|\nabla'^2 u(\cdot,z,\cdot)|\rangle\lesssim \sup_{z}\langle|f(\cdot,z,\cdot)|\rangle \,,
		   \end{equation}               
		   \begin{equation}\label{A1.4}
		    \sup_{z}\langle|\nabla'\partial_z u(\cdot,z,\cdot)|\rangle\lesssim \sup_{z}\langle|f(\cdot,z,\cdot)|\rangle \,.
		   \end{equation} 
		   In order to bound the off-diagonal components 
		   of the Hessian, we consider the 
		   decomposition
		   \begin{equation}\label{SPLIT}
		    u=u_N+u_C,
		   \end{equation} 
		 where $u_N$ solves   
		  \begin{equation}\label{A2}
                        \left\{\begin{array}{rclc}
		        (\partial_t-\Delta)u_N &=&f \qquad & {\rm for } \quad   z>0\,,\\
		        \partial_z u_N &=& 0 \qquad & {\rm for } \quad    z=0\,,\\
		          u_N &=& 0 \qquad & {\rm for } \quad    t=0\,,\\
		        \end{array}\right.        
                   \end{equation}
                     and $u_C$ solves
		   \begin{equation}\label{A3}
                        \left\{\begin{array}{rclc}
		        (\partial_t-\Delta)u_C&=&0   \qquad & {\rm for } \quad   z>0\,,\\
		        \partial_z u_C &=&\partial_z u  \qquad & {\rm for } \quad    z=0\,,\\
		        u_C &=&0  \qquad & {\rm for } \quad    t=0\,.\\
		        \end{array}\right.        
                   \end{equation}                       		       
		The splitting (\ref{SPLIT}) is valid by the uniqueness of the Neumann problem.
                For the auxiliary problems  (\ref{A2}) and (\ref{A3}) we have  the following bounds      		               
		\begin{equation}\label{A2.1}
		  \int\langle|\nabla'\partial_z u_N(\cdot,z,\cdot)|\rangle \frac{dz}{z}\lesssim \int\langle|f(\cdot,z,\cdot)|\rangle \frac{dz}{z}\,,
		\end{equation}       
       		\begin{equation}\label{A3.1}
		  \sup_z\langle|\nabla'\partial_z u_C(\cdot,z,\cdot)|\rangle\lesssim \langle|\nabla'\partial_z u(\cdot,z,\cdot)|_{z=0}\rangle \,.
		\end{equation}       
		We claim that estimates (\ref{A1.1}), (\ref{A1.2}),(\ref{A1.3}), (\ref{A1.4}), (\ref{A2.1}) and (\ref{A3.1})  yield (\ref{4}). 
		 \newline
		Let us first consider the bound for $\nabla'^2$. 
		Consider $u=u_1+u_2$, where $u_1$ and $u_2$ satisfy (\ref{A1}) with right hand side $f_1$ and $f_2$ respectively. 
		 We have	 
		\begin{eqnarray*}
		||\nabla'^2 u||_{(0,\infty)}&\lesssim &
		\sup_z\langle|\nabla'^2u_{1}|\rangle+\int\langle|\nabla'^2 u_{2}|\rangle\frac{dz}{z}\\
		& \stackrel{ (\ref{A1.1}) \& (\ref{A1.3})}{\lesssim} & \sup_z\langle|f_1|\rangle+\int\langle|f_2|\rangle\frac{dz}{z},
		\end{eqnarray*}		
		which implies, upon taking infimum over all decompositions $f=f_1+f_2$
		\begin{equation}\label{AAA}
		 ||\nabla'^2 u||_{(0,\infty)}\lesssim ||f||_{(0,\infty)}.
		\end{equation}      
		 We now consider a further decomposition of $u_2$ , i.e $u_2=u_{2C}+u_{2N}$ where $u_{2C}$ satisfies (\ref{A3}) and $u_{2N}$ satisfies (\ref{A2}).
		 Therefore $u=u_1+u_{2C}+u_{2N}$ and we can bound the off-diagonal components of the Hessian 
                \begin{eqnarray*}
		||\nabla'\partial_z u||_{(0,\infty)}& \lesssim &
		 \sup_z\langle|\nabla'\partial_z u_1|\rangle+\sup_z\langle|\nabla'\partial_zu_{2C}|\rangle+\int\langle|\nabla'\partial_z u_{2N}|\rangle\frac{dz}{z}\\
		& \stackrel{(\ref{A1.2}),(\ref{A3.1}),(\ref{A2.1}) \& (\ref{A1.4})}{\lesssim} & \sup_z\langle|f_1|\rangle+\int\langle|f_2|\rangle\frac{dz}{z}\,.
		\end{eqnarray*}		
		From the last inequality, passing to the infimum over all the possible decompositions of $f$ we get
		\begin{equation}\label{BBB}
		||\nabla'\partial_z u||_{(0,\infty)}\lesssim ||f||_{(0,\infty)}.
		\end{equation}         		
		On one hand estimate (\ref{AAA}) and (\ref{BBB}) imply 
		$$||\nabla\nabla' u||_{(0,\infty)}\lesssim||\nabla'^2 u||_{(0,\infty)}+||\nabla'\partial_z u||_{(0,\infty)}\,,$$   		
		on the other hand equation (\ref{III}) and  estimate (\ref{AAA}) yield         
		$$||(\partial_t-\partial_z^2) u||_{(0,\infty)}\lesssim ||f||_{(0,\infty)}\,.$$		               
		   \underline{Argument for \ref{A1.1}}
		   \newline
		   Let $u$ be a solution of problem of (\ref{A1}).
                   Keeping in mind Remark (\ref{TIME}) it is enough to show              
                   $$\int_0^{\infty}\int_0^{\infty}\langle|\nabla'^2u|\rangle'\frac{dz}{z}dt\lesssim \int_{0}^{\infty}\int_0^{\infty}\langle|f|\rangle'\frac{dz}{z}dt \,.$$    
		   By the Duhamel's principle we have     
		   \begin{equation}\label{Duhamel1}
                    u(x',z,t)=\int_{s=0}^{t} u_{s}(x',z,t)ds ,
                   \end{equation}
                   where $u_{s}$ is the solution to the homogeneous, initial value problem		   		  
	           \begin{equation}\label{Du-eq}	        
                   \left\{\begin{array}{rclc}
		   (\partial_t-\Delta)u_s&=&0  \qquad & {\rm for } \quad  z>0, t>s\,,\\
		    u_s&=&0  \qquad & {\rm for } \quad  z=0, t>s\,,\\
		    u_s&=&f  \qquad & {\rm for } \quad  z>0, t=s\,.\\
		    \end{array}\right.        
                   \end{equation}                                       
		 Extending $u$ and $f$ to the whole space by odd reflection 
		 \footnote{with abuse of notation we will call again $u$ and $f$ these extensions.},
		 we are left to study the problem       
		 \begin{equation*}
                 \left\{\begin{array}{rclc}
		 (\partial_t-\Delta)u_s&=&0 \qquad & {\rm for } \quad z\in \R, t>s\,,\\
		 u_s&=&f  \qquad & {\rm for } \quad z\in \R,  t=s\,,\\
	         \end{array}\right.        
                 \end{equation*}                    
		 the solution of which can be represented via heat kernel as     
                  \begin{equation}\label{representation}
		 \begin{array}{rclc}
		 u_s(x',z,t)&=&\int_{\R}\Gamma(\cdot,z-\tilde z, t-s)\ast_{x'}f(\cdot,\tilde z,s)d\tilde z\\
		 &=&\int_{0}^{\infty}\left[\Gamma(\cdot,z-\tilde z,t-s)-\Gamma ( \cdot,z+\tilde z,t-s)\right]\ast_{x'}f(\cdot,\tilde z,s)d\tilde z\,.\\
		 \end{array} 
                  \end{equation} 
		 The application of  $\nabla'^2$ to the representation above yields 
		 \begin{eqnarray*}
                  &&\nabla'^2u_s(x',z,t)\\
 		 &=&\footnotesize{\begin{cases}
 		     \int_{0}^{\infty}\int_{\R^{d-1}}\nabla'\Gamma_{d-1}(x'-\tilde{x'},t-s)\left(\Gamma_1(z-\tilde z,t-s)-\Gamma_1( z+\tilde z,t-s)\right)\nabla'f(\tilde{x'},\tilde z,s)d\tilde{x'}d\tilde z\,,\\
 		     \int_{0}^{\infty}\int_{\R^{d-1}}\nabla'^3\Gamma_{d-1}(x'-\tilde{x'},t-s)\left(\Gamma_1(z-\tilde z,t-s)-\Gamma_1 (z+\tilde z,t-s)\right)(-\Delta')^{-1}\nabla'f(\tilde{x'},\tilde z,s)d\tilde{x'}d\tilde z\,.\\
 		    \end{cases}}
 		    \end{eqnarray*}    		
 		Averaging in the horizontal direction we obtain, on the one hand
%
		{\footnotesize{
		 \begin{eqnarray*}
                  &&\langle|\nabla'^2u_s(\cdot,z,t)|\rangle'\\
                  &\lesssim&\int_{0}^{\infty}\langle|\nabla'\Gamma_{d-1} (\cdot,t-s)|\rangle'|\Gamma_1(z-\tilde z,t-s)-\Gamma_1 ( z+\tilde z,t-s)|\langle|\nabla'f(\cdot,\tilde z,s)|\rangle' d\tilde z\\
                  &\stackrel{(\ref{z0})\&(\ref{BAND2})}{\lesssim}&\int_{0}^{\infty}\frac{1}{(t-s)^{\frac{1}{2}}}|\Gamma_1(z-\tilde z,t-s)-\Gamma_1 ( z+\tilde z,t-s)|\frac{1}{R}\langle|f(\cdot,\tilde z,s)|\rangle' d\tilde z\\
		 \end{eqnarray*}}}
		 and, on the other hand
		  {\footnotesize{
		 \begin{eqnarray*}
                  &&\langle|\nabla'^2u_s(\cdot,z,t)|\rangle'\\
                   &\lesssim& \int_{0}^{\infty}\langle|\nabla'^3\Gamma_{d-1} (\cdot,t-s)\rangle'|\Gamma_1(z-\tilde z,t-s)-\Gamma_1( z+\tilde z,t-s)|\langle|(-\Delta')^{-1}\nabla'f(\cdot,\tilde z,s)|\rangle'd\tilde z\,\\
                   &\stackrel{(\ref{z0})\&(\ref{BAND1})}{\lesssim}&\int_{0}^{\infty}|\Gamma_1(z-\tilde z,t-s)-\Gamma_1( z+\tilde z,t-s)|\frac{1}{(t-s)^{\frac{3}{2}}} R\langle|f(\cdot,\tilde z,s)|\rangle'd\tilde z\,.
		 \end{eqnarray*}}}		 
		 Multiplying by the weight $\frac{1}{z}$ and integrating in $z\in(0,\infty)$ we get
		 \begin{equation*}
		\int_0^{\infty}\langle|\nabla'^2u_s(\cdot,t)|\rangle'\frac{dz}{z}\lesssim \left(\sup_{\tilde z}\int_0^{\infty}K_{t-s}(z,\tilde z) dz\right)
		 \footnotesize{\begin{cases}
		\frac{1}{(t-s)^{\frac{1}{2}}}\frac{1}{R}\int_{0}^{\infty}\langle|f(\cdot,\tilde z,s)|\rangle' \frac{d\tilde z}{\tilde z}\,,\\
		\frac{R}{(t-s)^{\frac{3}{2}}}\int_{0}^{\infty} \langle |f(x',\tilde z,s)|\rangle'\frac{d\tilde z}{\tilde z}\,,
		\end{cases}}
		\end{equation*}    
		where we called $K_{t-s}(z,\tilde z)=\frac{\tilde z}{z}|\Gamma_1(z-\tilde z,t-s)-\Gamma_1( z+\tilde z,t-s)|$.
		\newline
		 From Lemma \ref{Lemma5} we infer 
		 $$\sup_{\tilde z}\int_0^{\infty}K_{t-s}(z,\tilde z) dz\stackrel{(\ref{EE1})}{\lesssim} \int_{\R}|\Gamma_1(z,t-s)|dz+\sup_{z\in \R}(z^2|\partial_z\Gamma_1(z,t-s)|)\stackrel{(\ref{x1})\&(\ref{y3})}{\lesssim}1$$
                and therefore we have
		\begin{equation*}
		\int_0^{\infty}\langle|\nabla'^2u_s(\cdot,z,t)|\rangle'\frac{dz}{z}\lesssim
		 \footnotesize{\begin{cases}
		\frac{1}{(t-s)^{\frac{1}{2}}}\frac{1}{R}\int_{0}^{\infty}\langle|f(\cdot,\tilde z,s)|\rangle' \frac{d\tilde z}{\tilde z}\,,\\
		\frac{1}{(t-s)^{\frac{3}{2}}} R \int_{0}^{\infty}\langle |f(\cdot,\tilde z,s)|\rangle'\frac{d\tilde z}{\tilde z}\,.
		\end{cases}}
		\end{equation*}    
		Finally, inserting the previous estimate into the Duhamel formula (\ref{Duhamel1} )and integrating in time we get
               \begin{eqnarray}
               &&\int_0^{\infty}\langle|\nabla'^2 u(\cdot,z,t)|\rangle'\frac{dz}{z} dt \notag \\
               &\stackrel{\ref{Duhamel1}}{\lesssim}&\int_0^{\infty}\int_0^t\langle|\nabla'^2 u_s(\cdot,z,t)|\rangle' \frac{dz}{z} ds dt\notag\\
               &\lesssim&\int_0^{\infty}\int_s^{\infty} \min\{\frac{1}{R(t-s)^{\frac{1}{2}}},\frac{R}{(t-s)^{\frac{3}{2}}}\}\int_0^{\infty}\langle|f(\cdot,\tilde z,s)|\rangle'\frac{d\tilde z}{\tilde z} dt ds\notag\\
               &\lesssim&\int_0^{\infty}\int_s^{\infty}\min\{\frac{1}{R(t-s)^{\frac{1}{2}}},\frac{R}{(t-s)^{\frac{3}{2}}}\}dt\int_0^{\infty}\langle|f(\cdot,\tilde z,s)|\rangle'\frac{d\tilde z}{\tilde z} ds \label{Es1}\\
               &\lesssim&\int_0^{\infty}\int_0^{\infty}\min\{\frac{1}{R\tau^{\frac{1}{2}}},\frac{R}{\tau^{\frac{3}{2}}}\}d\tau\int_0^{\infty}\langle|f(\cdot,\tilde z,s)|\rangle'\frac{d\tilde z}{\tilde z} ds\label{Es2},\\
               &\lesssim&\int_0^{\infty}\int_0^{\infty}\langle|f(\cdot,\tilde z,s)|\rangle'\frac{d\tilde z}{\tilde z} ds,\notag
               \end{eqnarray}
               where in the second to last inequality we used 
                \begin{equation}\label{MIN}
                \int_0^{\infty}\min\left\{\frac{1}{R\tau^{\frac{1}{2}}},\frac{R}{\tau^{\frac{3}{2}}}\right\}d\tau\lesssim 1\,.
                \end{equation}                 
		 \underline{Argument for \ref{A1.2}:}  
		 \newline
		     Let $u$ be a solution of problem of (\ref{A1}). Recall that we need to prove
		  
		   \begin{equation}\label{BO1}
		      \int_0^{\infty}\langle|\nabla'\partial_z u|_{z=0}(\cdot,z,t)|\rangle' dt\lesssim \int_0^{\infty}\int_0^{\infty}\langle|f(\cdot,z,t)|\rangle dt\frac{dz}{z}\,.
		   \end{equation}    		   
                  The solution of the equation (\ref{Du-eq}) extended to the whole space by odd reflection can be represented by (\ref{representation}) (see argument for (\ref{A1.1})). Therefore
		  \begin{eqnarray*}		
		  &&\nabla'\partial_zu_s(x',z,t)|_{z=0}\\  
		  &=&\begin{cases}
		  -2\int_{\R^{d-1}}\int_0^{\infty} \Gamma_{d-1}(x'-\tilde{x'},t-s)\partial_z\Gamma_1(\tilde z,t-s)\nabla'f(\tilde {x'},\tilde{z},s)d\tilde{x'}d\tilde z\,,\\
		  -2\int_{\R^{d-1}}\int_0^{\infty}\nabla'\Gamma_{d-1}(x'-\tilde{x'},t-s) \partial_z\Gamma_1(\tilde z,t-s)\nabla'(-\Delta')^{-1}\nabla'f(\tilde{x'},\tilde{z},s) d\tilde{x'}d\tilde z\,.
		  \end{cases}
		  \end{eqnarray*}          
		  Taking the horizontal average we get, on the one hand
		  \begin{eqnarray*}
		    &&\langle|\nabla'\partial_zu_s(\cdot,z,t)|_{z=0}|\rangle'\\
		    &\lesssim& \int_0^{\infty}\langle|\Gamma_{d-1}(\cdot,t-s)|\rangle'|\partial_z \Gamma_1(\tilde z,t-s)|\langle|\nabla'f(\cdot,\tilde z,s)|\rangle' d\tilde z\\
		    &\stackrel{(\ref{z0})}{\lesssim}&\int_0^{\infty}|\partial_z \Gamma_1(\tilde z,t-s)|\langle|\nabla' f(\cdot,\tilde z,s)|\rangle'd\tilde z\\
		    &\stackrel{(\ref{BAND2})}{\lesssim}&\frac{1}{R}\int_0^{\infty}|\partial_z \Gamma_1(\tilde z,t-s)|\langle|f(\cdot,\tilde z,s)|\rangle'd\tilde z\\
		    &\lesssim&\frac{1}{R}\sup_{\tilde z}|\tilde z\partial_z\Gamma_1(\tilde z,t-s)|\int_0^{\infty}\langle|f(\cdot,\tilde{z},s)|\rangle'\frac{d\tilde z}{\tilde z}
		  \end{eqnarray*}          
		  and on the other hand        
		  \begin{eqnarray*}
		    &&\langle|\nabla'\partial_zu_s(\cdot,z,t)|_{z=0}|\rangle'\\
		    &\lesssim& \int_0^{\infty}\langle|(\nabla')^2\Gamma_{d-1}(\cdot,t-s)|\rangle'|\partial_z \Gamma_1(\tilde z,t-s)|\langle|(-\Delta')^{-1}\nabla'f(\cdot,\tilde z,s)|\rangle' d\tilde z\\
		    &\stackrel{(\ref{z0})}{\lesssim}&\frac{1}{(t-s)}\int_0^{\infty}|\partial_z \Gamma_1(\tilde z,t-s)|\langle|(-\Delta')^{-1}\nabla' f(\cdot,\tilde z,s)|\rangle'd\tilde z\\
		    &\stackrel{(\ref{BAND1})}{\lesssim}&\frac{R}{(t-s)}\int_0^{\infty}|\partial_z \Gamma_1(\tilde z,t-s)|\langle|f(\cdot,\tilde z,s)|\rangle'd\tilde z\\
		    &\lesssim&\frac{R}{(t-s)}\sup_{\tilde z}|\tilde z\partial_z\Gamma_1(\tilde z,t-s)|\int_0^{\infty}\langle|f(\cdot,\tilde{z},s)|\rangle'\frac{d\tilde z}{\tilde z}\,.
		  \end{eqnarray*}                 
		Using the estimate (\ref{y2}) we get
		\begin{equation*}
		 \langle|\nabla'\partial_zu_s(x',z,t)|_{z=0}|\rangle' \lesssim \begin{cases}
		                                                                \frac{1}{(t-s)^{1/2}R}\int_0^{\infty}\langle|f(\cdot,\tilde z,s)|\rangle'\frac{d\tilde z}{\tilde z}\,,\\
		                                                                \frac{R}{(t-s)^{3/2}}\int_0^{\infty}\langle|f(\cdot,\tilde z,s)|\rangle'\frac{d\tilde z}{\tilde z}\,.
		                                                               \end{cases}
		                                                               \end{equation*} 		                                                                     
               Finally,  inserting into Duhamel's formula and integrating in time we have 
               \begin{eqnarray*}
               &&\int_0^{\infty}\langle|\nabla'\partial_z u(\cdot,z,t)|_{z=0}\rangle' dt \\
               &\stackrel{(\ref{Duhamel1})}{\lesssim}&\int_0^{\infty}\int_0^t\langle|\nabla'\partial_z u_s(\cdot,z,t)|_{z=0}\rangle' ds dt\\
               &\lesssim& \int_0^{\infty}\int_s^{\infty}\min\{\frac{1}{R(t-s)^{\frac{1}{2}}},\frac{R}{(t-s)^{\frac{3}{2}}}\}\int_0^{\infty}\langle|f(\cdot,\tilde z,s)|\rangle'\frac{d\tilde z}{\tilde z} dt ds\\
               &\stackrel{(\ref{Es1})\& (\ref{Es2})}{\lesssim}&\int_0^{\infty}\int_0^{\infty}\langle|f(x',z,s)|\rangle'\frac{d\tilde z}{\tilde z} ds.
               \end{eqnarray*}          
                 \underline{Argument for (\ref{A1.3}):}
                 \newline
                  Let $u$ be the solution of problem (\ref{A1}). 
                  We recall that we want to prove               
                  \begin{equation}\label{SUP}
                  \sup_z\int_0^{\infty}\langle|\nabla'^2u(\cdot,z,t)|\rangle'dt\lesssim \sup_z\int_{0}^{\infty}\langle|f(\cdot,z,t)|\rangle'dt \,.
                  \end{equation}   		   		  
                  The solution of equation (\ref{Du-eq}) extended to the whole space  can be represented by (\ref{representation})
                  (see argument for (\ref{A1.1})). Therefore
                  applying $\nabla'^2$ to (\ref{representation})and considering the horizontal average we have, on the one hand 
                \begin{eqnarray*}
 		 &&\langle|\nabla'^2u_s(\cdot,z,t)|\rangle'\\
 		&\lesssim& \int_{\R}\langle|\nabla'\Gamma_{d-1}(\cdot,t-s)|\rangle'|\Gamma_1(z-\tilde z,t-s)|\langle|\nabla'f(\cdot,\tilde{z},s)|\rangle' d\tilde z\\
                &\stackrel{(\ref{z0})\& (\ref{BAND2})}{\lesssim}& \int_{\R}\frac{1}{(t-s)^{\frac{1}{2}}}|\Gamma_1(z-\tilde z,t-s)|\frac{1}{R}\langle|f(\cdot,\tilde{z},s)|\rangle' d\tilde z
                \end{eqnarray*}                
                and on the other hand 
                 \begin{eqnarray*}
 		 &&\langle|\nabla'^2u_s(\cdot,z,t)|\rangle'\\
 		 &\lesssim&\int_{\R}\langle|\nabla'^3\Gamma_{d-1}(\cdot,t-s)|\rangle'|\Gamma_1(z-\tilde z,t-s)|\langle|(-\Delta')^{-1}\nabla'f(\cdot,\tilde{z},s)|\rangle' d\tilde z\,\\ 
 		 &\stackrel{(\ref{z0}) \& (\ref{BAND1})}{\lesssim} &\int_{\R}\frac{1}{(t-s)^{\frac{3}{2}}}|\Gamma_1(z-\tilde z,t-s)|R\langle|f(\cdot,\tilde{z},s)|\rangle' d\tilde z\,.
		 \end{eqnarray*}                 
                  Inserting the above estimates in the Duhamel's formula (\ref{Duhamel1}), we have                 
                 {\footnotesize{\begin{eqnarray*}
                 &&\int_0^{\infty}\int_0^t \langle|\nabla'^2u_s(z,\cdot)|\rangle'dsdt\\
                 &\lesssim& \int_0^{\infty}\int_s^{\infty}\min\left\{\frac{1}{R(t-s)^{\frac{1}{2}}},\frac{R}{(t-s)^{\frac{3}{2}}}\right\}\int_{\R}|\Gamma_1(z-\tilde{z},t-s)|\langle|f(\cdot,\tilde{z},s)|\rangle'd\tilde z dsdt\\
                 &\lesssim&\int_{\R}\left(\int_0^{\infty}\min\left\{\frac{1}{R \tau^{\frac{1}{2}}},\frac{R}{\tau^{\frac{3}{2}}}\right\}|\Gamma_1(z-\tilde{z},\tau)|d\tau\right)\int_0^{\infty}\langle|f(\cdot,\tilde{z},s)|\rangle'dsd\tilde z\\
		 &\lesssim& \sup_{\tilde{z}}\int_0^{\infty}\langle|f(\cdot,\tilde{z},s)|\rangle'ds\int_{\R}\int_0^{\infty}\min\left\{\frac{1}{R \tau^{\frac{1}{2}}},\frac{R}{\tau^{\frac{3}{2}}}\right\}|\Gamma_1(z-\tilde{z},\tau)|d\tau d\tilde z\\
		 &\stackrel{(\ref{x1})}{\lesssim}&\sup_{\tilde{z}}\int_0^{\infty}\langle|f(\cdot,\tilde{z},s)|\rangle'ds\int_0^{\infty}\min\left\{\frac{1}{R \tau^{\frac{1}{2}}},\frac{R}{\tau^{\frac{3}{2}}}\right\}d\tau\int_{\R}|\Gamma_1(z-\tilde{z},\tau)|d\tilde z\\
		 &\stackrel{(\ref{MIN})}{\lesssim}&\sup_{\tilde{z}}\int_0^{\infty}\langle|f(\cdot,\tilde{z},s)|\rangle'ds\,.
                 \end{eqnarray*}}}
                 Taking the supremum in $z$ we obtain the desired estimate.          
                 \newline       
                  \underline{Argument for (\ref{A1.4}):}
                  \newline
                  Let $u$ be the solution of problem (\ref{A1}). 
                   We claim               
                   \begin{equation}\label{SUP2}
                   \sup_z\int_0^{\infty}\langle|\nabla'\partial_zu|\rangle'dt\lesssim \sup_z\int_{0}^{\infty}\langle|f|\rangle'dt \,.
                   \end{equation}   		 
		  The solution of the equation (\ref{Du-eq}) extended to the whole space  can be represented by (see argument for (\ref{A1.1}))      
		 \begin{align*}
		  u_{s}(x',z,t)&=\int_{\R}\Gamma(\cdot,z-\tilde z, t-s)\ast_{x'}f(\cdot,\tilde z,s)d\tilde z\,.
		 \end{align*}          
		 Applying $\nabla'\partial_z$ and considering the horizontal average we obtain, on the one hand 		 
                 \begin{eqnarray*}
 		 &&\langle|\nabla'\partial_z u_s(\cdot,z,t)|\rangle'\\ 		 
 		 &\lesssim&\int_{\R}\langle|\Gamma_{d-1}(\cdot,t-s)|\rangle'|\partial_z\Gamma_1(z-\tilde z,t-s)|\langle|\nabla'f(\cdot,\tilde{z},s)|\rangle' d\tilde z\\
 		 &\stackrel{\ref{BAND2}}{\lesssim}&\int_{\R}|\partial_z\Gamma_1(z-\tilde z,t-s)|\frac{1}{R}\langle|f(\cdot,\tilde{z},s)|\rangle' d\tilde z\\
 		 \end{eqnarray*}
 		 and, on the other hand 
 		 \begin{eqnarray*}
 		 &&\langle|\nabla'\partial_z u_s(\cdot,z,t)|\rangle'\\ 		 
 		 &\lesssim&\int_{\R}\langle|\nabla'^2\Gamma_{d-1}(\cdot,t-s)|\rangle'|\partial_z\Gamma_1(z-\tilde z,t-s)|\langle|(-\Delta')^{-1}\nabla'f(\cdot,\tilde{z},s)|\rangle' d\tilde z\\  
 		 &\stackrel{\ref{BAND1}}{\lesssim}&\int_{\R}\frac{1}{(t-s)}|\partial_z\Gamma_1(z-\tilde z,t-s)|R\langle|f(\cdot,\tilde{z},s)|\rangle' d\tilde z\,.\\  
 		 \end{eqnarray*}		 
                 Inserting the above estimates in the Duhamel's formula (\ref{Duhamel1}), we have                 
                 {\footnotesize{\begin{eqnarray*}
                 &&\int_0^{\infty}\int_0^t \langle|\nabla'\partial_z u_s(z,\cdot)|\rangle'dsdt\\
                 &\lesssim& \int_0^{\infty}\int_s^{\infty}\min\left\{\frac{1}{R},\frac{R}{(t-s)}\right\}\int_{\R}|\partial_z\Gamma_1(z-\tilde{z},t-s)|\langle|f(\cdot,\tilde{z},s)|\rangle'd\tilde z dtds\\
                 &\lesssim&\int_{\R}\left(\int_0^{\infty}\min\left\{\frac{1}{R},\frac{R}{\tau}\right\}|\partial_z\Gamma_1(z-\tilde{z},\tau)|d\tau\right)\int_0^{\infty}\langle|f(\cdot,\tilde{z},s)|\rangle'dsd\tilde z\\
		 &\lesssim& \sup_{\tilde{z}}\int_0^{\infty}\langle|f(\cdot,\tilde{z},s)|\rangle'ds\int_{\R}\int_0^{\infty}\min\left\{\frac{1}{R },\frac{R}{\tau}\right\}|\partial_z\Gamma_1(z-\tilde{z},\tau)|d\tau d\tilde z\\
		 &\stackrel{(\ref{x1})}{\lesssim}&\sup_{\tilde{z}}\int_0^{\infty}\langle|f(\cdot,\tilde{z},s)|\rangle'ds\int_0^{\infty}\min\left\{\frac{1}{R \tau^{\frac{1}{2}}},\frac{R}{\tau^{\frac{3}{2}}}\right\}d\tau\\
		 &\stackrel{(\ref{MIN})}{\lesssim}&\sup_{\tilde{z}}\int_0^{\infty}\langle|f(\cdot,\tilde{z},s)|\rangle'ds\,.
                 \end{eqnarray*}}}
                 Taking the supremum in $z$ we obtain the desired estimate. 
                 \newline           
		\underline{Argument for (\ref{A2.1})}
                 \newline
		 We recall that we want to show                
                   $$\int_0^{\infty}\int_0^{\infty}\langle|\nabla'\partial_zu_N|\rangle'\frac{dz}{z}dt\lesssim \int_{0}^{\infty}\int_0^{\infty}\langle|f|\rangle'\frac{dz}{z}dt ,$$    
		   where $u_N$ be the solution to the non-homogeneous heat equation with Neumann boundary conditions (\ref{A2}).
		   By the Duhamel's principle we have     
		    $$u_N(x',z,t)=\int_{s=0}^{t} u_{N_s}(x',z,t)ds, $$    
		    where  $u_{N_s}$ is solution to            
	          \begin{equation*}
                  \left\{\begin{array}{rclc}
		  (\partial_t-\Delta)u_{N_s}&=&0 \qquad & {\rm for } \quad  z>0, t>s\,,\\
		  \partial_zu_{N_s}&=&0  \qquad & {\rm for } \quad z=0, t>s\,,\\
		  u_{N_s}&=&f  \qquad & {\rm for } \quad z>0, t=s\,,\\
		  \end{array}\right.        
                  \end{equation*}                    
                  is the solution of problem (\ref{A1}).        
		 Extending this equation to the whole space by even reflection
		 \footnote{With abuse of notation we will denote with $u_{N_s}$ and $f$ their even reflection}, we are left to study the problem        
		 \begin{equation*}
                  \left\{\begin{array}{rclc}
		  (\partial_t-\Delta)u_{N_s}&=&0 \qquad & {\rm for } \quad z\in \R, t>s\,,\\
		  u_{N_s}&=&f \qquad & {\rm for } \quad  t=s\,,\\
		  \end{array}\right.        
                 \end{equation*}        
		the solution of which can be  represented via heat kernel as       
		\begin{align*}
		  u_{N_s}(x',z,t)&=\int_{\R}\Gamma(\cdot,z-\tilde z, t-s)\ast_{x'}f(\cdot,\tilde z,s)d\tilde z\\
		  &=\int_{0}^{\infty}\left[\Gamma(\cdot,\tilde z+z,t-s)+\Gamma (\cdot,\tilde z-z,t-s)\right]\ast_{x'}f(\cdot,\tilde z,s)d\tilde z\,.
		\end{align*}          		 
		 Applying $\nabla'\partial_z$ to the representation above
 		 \begin{eqnarray*}
                  &&\nabla'\partial_zu_{N_{s}}(x',z,t)\\
		 &=&\footnotesize{
 		 \begin{cases}
 		     \int_{0}^{\infty}\int_{\R^{d-1}}\Gamma_{d-1}(x'-\tilde{x'},t-s)\left(\partial_z\Gamma_1(\tilde z+z,t-s)-\partial_z\Gamma_1 ( \tilde z-z,t-s)\right)\nabla'f(\tilde{x'},\tilde z,s)d\tilde{x'}d\tilde z\,,\\
 		     \int_{0}^{\infty}\int_{\R^{d-1}}\nabla'^2\Gamma_{d-1}(x'-\tilde{x'},t-s)\left(\partial_z\Gamma_1(\tilde z+z,t-s)-\partial_z\Gamma_1 (\tilde z-z,t-s)\right)(-\Delta')^{-1}\nabla'f(\tilde{x'},\tilde z,s)d\tilde{x'}d\tilde z\,
 		    \end{cases}}
 		 \end{eqnarray*}    		
 		and averaging in the horizontal direction we obtain, on the one hand
%
                {\footnotesize{  
		\begin{eqnarray*}
                 &&\langle|\nabla'\partial_zu_{N_s}(\cdot,z,t)|\rangle'\\
                 &\lesssim&\int_{0}^{\infty}\langle|\Gamma_{d-1}(\cdot,t-s)|\rangle'|\partial_z\Gamma_1(\tilde z+z,t-s)-\partial_z\Gamma_1 (\tilde z-z,t-s)|\langle|\nabla'f(\cdot,\tilde z,s)|\rangle' d\tilde z\\
                 &\stackrel{(\ref{z0})\&(\ref{BAND2})}{\lesssim}&\frac{1}{R}\int_{0}^{\infty}|\partial_z\Gamma_1(\tilde z+z,t-s)-\partial_z\Gamma_1 (\tilde z-z,t-s)|\langle|f(\cdot,\tilde z,s)|\rangle' d\tilde z
 		 \end{eqnarray*}    }}
 		 and, on the other hand 
		  {\footnotesize{
 		\begin{eqnarray*}
                 &&\langle|\nabla'\partial_zu_{N_s}(\cdot,z,t)|\rangle'\\
 		 &\lesssim& \int_{0}^{\infty}\langle|\nabla'^2\Gamma_{d-1}(\cdot,t-s)|\rangle'|\partial_z\Gamma_1(\tilde z+z,t-s)-\partial_z\Gamma_1(\tilde z-z,t-s)|\langle|(-\Delta')^{-1}\nabla'f(\cdot,\tilde z,s)\rangle'd\tilde z\\
 		 &\stackrel{(\ref{z0})\&(\ref{BAND1})}{\lesssim}&\frac{R}{(t-s)} \int_{0}^{\infty}|\partial_z\Gamma_1(\tilde z+z,t-s)-\partial_z\Gamma_1(\tilde z-z,t-s)|\langle|f(\cdot,\tilde z,s)|\rangle'd\tilde z\,.
 		 \end{eqnarray*}  }}
		 Multiplying by the weight $\frac{1}{z}$ and integrating in $z\in(0,\infty)$ we get
		 \begin{equation*}
		\int_0^{\infty}\langle|\nabla'\partial_zu_{N_s}(\cdot,z,t)|\rangle'\frac{dz}{z}\lesssim \sup_{\tilde z}\int_0^{\infty}K_{t-s}(z,\tilde z) dz
		\begin{cases}
		\frac{1}{R}\int_{0}^{\infty}\langle|f(\cdot,\tilde z,s)|\rangle' \frac{d\tilde z}{\tilde z}\,,\\
		\frac{1}{(t-s)} R\int_{0}^{\infty}\langle |f(\cdot,\tilde z,s)|\rangle'\frac{d\tilde z}{\tilde z}\,,
		\end{cases}
		\end{equation*}    
		where we called $K_{t-s}(z,\tilde z)=\frac{\tilde z}{z}|\partial_z\Gamma_1(\tilde z-z,t-s)-\partial_z\Gamma_1( z+\tilde z,t-s)|$.
		\newline
		 Recalling 
		 $$\sup_{\tilde z}\int_0^{\infty}K_{t-s}(z,\tilde z) dz\stackrel{(\ref{EE1})}{\lesssim} \int_{\R}|\partial_z\Gamma_1(z,t-s)|dz+\sup_{z\in \R}(z^2|\partial^2_z\Gamma_1(z,t-s)|)$$
		 and observing that, in this case         
		 $$\int_{\R}|\partial_z \Gamma_1(z,t-s)|dz+\sup_{z\in \R}(z^2|\partial_z\Gamma_1(z,t-s)|)\stackrel{(\ref{x1})\&(\ref{y3})}{\lesssim}\frac{1}{(t-s)^{\frac{1}{2}}} ,$$
		 we can conclude that      		
		\begin{equation*}
		\int_0^{\infty}\langle|\nabla'\partial_zu_{N_s}(\cdot,t)|\rangle'\frac{dz}{z}\lesssim
		 \footnotesize{\begin{cases}
		\frac{1}{(t-s)^{\frac{1}{2}}}\frac{1}{R}\int_{0}^{\infty}\langle|f(\cdot,\tilde z,s)|\rangle' \frac{d\tilde z}{\tilde z}\\
		\frac{1}{(t-s)^{\frac{3}{2}}} R \int_{0}^{\infty}\langle |f(\cdot,\tilde z,s)|\rangle'\frac{d\tilde z}{\tilde z}.
		\end{cases}}
		\end{equation*}    
		Finally,  inserting (\ref{Duhamel1}) and integrating in time we have 
               \begin{eqnarray*}
               &&\int_0^{\infty}\int_0^{\infty}\langle|\nabla'\partial_z u_{N_s}(\cdot,z,t)|\rangle'\frac{dz}{z} dt \\
               &\stackrel{(\ref{Duhamel1})}{\lesssim}&\int_0^{\infty}\int_0^{\infty}\int_0^t\langle|\nabla'\partial_z u_{N_s}(\cdot,\tilde{z},t)|\rangle' \frac{dz}{z} ds dt\\
               &\lesssim&\int_s^{\infty} \int_0^{\infty}\min\{\frac{1}{R(t-s)^{\frac{1}{2}}},\frac{R}{(t-s)^{\frac{3}{2}}}\}\int_0^{\infty}\langle|f(\cdot,\tilde{z},s)|\rangle'\frac{d\tilde z}{\tilde z} ds dt\\
               &\stackrel{(\ref{Es1})\&(\ref{Es2})}{\lesssim}&\int_0^{\infty}\int_0^{\infty}\langle|f(\cdot,\tilde{z},s)|\rangle'\frac{d\tilde z}{\tilde z} ds\,.
               \end{eqnarray*}
		  \underline{Argument for (\ref{A3.1}):}
                   \newline
		  Recall that we need to prove 
		  $$\sup_z\int_0^{\infty}|\nabla'\partial_zu_C|dt\lesssim \langle|\nabla'\partial_zu|_{z=0}\rangle'\,.$$
                  By  equation (\ref{A3}), the even extension $\overline{u_C}$ satisfies                  		          
		  \begin{equation}\label{BO}
		   (\partial_t-\Delta)\overline{u_C}=-[\partial_z\overline{u_C}]\delta_{z=0}=-2\partial_z u_C\delta_{z=0}=-2\partial_z u|_{z=0}\delta_{z=0}
		  \end{equation}		          
		  and therefore we study the following problem on the whole space
                 \begin{equation}\label{A3b}
                  \left\{\begin{array}{rclc}
		   (\partial_t-\Delta)\overline{u_C}&=&-2\partial_z u|_{z=0}\delta \qquad & {\rm for } \quad z\in \R, t>0\,,\\
		  \overline{u_C}&=&0\qquad & {\rm for } \quad  t=0\,.\\
		  \end{array}\right.        
                 \end{equation}                  
                   By Duhamel's principle 
                   \begin{equation}\label{DU2}\overline{u_{C}}(x',z,t)=\int_{s=0}^{t} \overline{u_{C_s}}(x',z,t)ds ,\end{equation}
                  where $\overline{u_{C_s}}$ solves the initial value problem                       
	          \begin{equation}\label{A3bb}
                  \left\{\begin{array}{rclc}
		  (\partial_t-\Delta)\overline{u_{C_s}}&=&0 \qquad & {\rm for } \quad  z\in \R, t>s\,,\\
		  \overline{u_{C_s}}&=&-2\partial_z u|_{z=0}\delta  \qquad & {\rm for } \quad z\in \R, t=s\,.\\
		  \end{array}\right.        
                  \end{equation}                                    
		 The solution of problem (\ref{A3bb}) can be represented via the heat kernel as         
		  \begin{eqnarray*}
		   \overline{u_{C_s}}(x',z, t)&=&\int \Gamma(z-\tilde z,t-s)\ast_{x'}(-2\partial_z u|_{z=0}\delta)(\tilde z,s) d\tilde z,\\
		    &=&-2\Gamma(z,t-s)\ast_{x'}\partial_z u(z, s)|_{z=0}\,.
		  \end{eqnarray*}   
		  We apply $\nabla'\partial_z$ to the representation above
		 \begin{equation*}
                 \nabla'\partial_z\overline{u_{C_s}}(x',z,t)=
                  \int_{\R^{d-1}}-2\Gamma_{d-1}(x'-\tilde{x'},t-s)\partial_z\Gamma_1(z,t-s)\nabla'\partial_z u(\cdot,z, s)|_{z=0}d\tilde{x'}
		 \end{equation*}    		   
		 and then average in the horizontal direction, 		  
		 \begin{eqnarray*}
                 &&\langle|\nabla'\partial_z\overline{u_{C_s}}(x',z,t)|\rangle'\\
                 &\lesssim&\langle|\Gamma_{d-1}(x',t-s)|\rangle'|\partial_z\Gamma_1(z,t-s)|\langle|\nabla'\partial_z u(\cdot,z, s)|_{z=0}|\rangle'\\
                 &\stackrel{\ref{z0}}{\lesssim}&  |\partial_z\Gamma_1(z,t-s)|\langle|\nabla'\partial_z u(\tilde{x'},z, s)|_{z=0}|\rangle'\,.
		 \end{eqnarray*}    
	         Inserting the previous estimate in the Duhamel formula \ref{DU2} and integrating in time we get
	         \begin{eqnarray}
	         &&\int_0^{\infty}\langle|\nabla'\partial_z\overline{u_{C}}(x',z,t)|\rangle'dt\notag\\
	         &\leq& \int_0^{\infty}\int_0^t\langle|\nabla'\partial_z\overline{u_{C_s}}(x',z,t)|\rangle'dsdt\notag\\
	         &\lesssim& \int_0^{\infty}\int_s^{\infty}|\partial_z\Gamma_1(z,t-s)|dt\langle|\nabla'\partial_z u(\tilde{x'},z, s)|_{z=0}|\rangle' ds\notag\\
	         &\stackrel{\ref{y1}}{\lesssim}&  \int_0^{\infty}\langle|\nabla'\partial_z u(\tilde{x'},z, s)|_{z=0}|\rangle' ds\,.\label{here}
	          \end{eqnarray} 
	          The estimate (\ref{A3.1}) follows immediately after passing to the supremum in (\ref{here}).
	          \end{proof}	 
          
\section{Appendix}\label{App}

\subsection{Preliminaries}
We start this section by proving some elementary bounds and equivalences, coming directly from the 
definition of horizontal bandedness. These will turn  to be crucial in the proof of the
main result.
     \begin{lemma}\label{Bandedness}\ \\
       \begin{enumerate}
         \item[a)]   If 
                     \begin{equation}\label{B1}
                     \F r(k',z,t)=0 \quad  {\rm{ unless }} \quad R|k'|\geq 4\,
                     \end{equation} then 
                     \begin{equation}\label{BAND1}
                     \langle|r(\cdot,z,t)| \rangle'\leq R\langle|\nabla' r(\cdot,z,t)|\rangle'\,.
                     \end{equation}

                    In particular
                    $$||r||_{(0,\infty)}\leq R ||\nabla' r||_{(0,\infty)}\,.$$
        
          \item[b)] If    
                    \begin{equation}\label{B2}
                    \F r(k',z,t)=0 \quad  {\rm{ unless }} \quad R|k'|\leq 1\,
                    \end{equation} then 
                    \begin{equation}\label{BAND2}
                    \langle|\nabla'r(\cdot,z,t)| \rangle'\leq \frac{1}{R}\langle| r(\cdot,z,t)|\rangle'\,.
                    \end{equation}
                    In particular
                    $$||\nabla'r||_{(0,\infty)}\leq R ||r||_{(0,\infty)}\,.$$

         \item[c)]  If    
                    $$\F r(k',z,t)=0 \quad  {\rm{ unless }} \quad 1\leq R|k'|\leq 4 $$ then
                    \begin{equation}\label{P}
                    ||\nabla'(-\Delta')^{-\frac{1}{2}}r||_{(0,\infty)}\sim ||r||_{(0,\infty)}\,,
                    \end{equation}
                    and 
                   \begin{equation}\label{Q}
                   ||(-\Delta')^{\frac{1}{2}}r||_{(0,\infty)}\sim||\nabla' r||_{(0,\infty)}\,.
                   \end{equation}        
        \end{enumerate}
       \end{lemma}
                  
      \begin{remark}
       All the results stated in Lemma \ref{Bandedness} are valid with the norm $||\cdot||_{(0,\infty)}$ replaced with $||\cdot||_{(0,1)}$.
      \end{remark}
      \begin{remark}
      Notice that from (\ref{P}) and (\ref{Q}), it follows
      \begin{equation}\label{R}
      ||\nabla'(-\Delta')^{-1}\nabla'\cdot r||_{(0,\infty)}\lesssim ||r||_{(0,\infty)}\,.
      \end{equation}  
      \end{remark}

\begin{proof}\ \\

  \begin{enumerate}
    \item[a)] By rescaling we may assume $R=1$.
 
	      Let $\phi\in\ES (\R^{d-1})$ be a Schwartz function such that 

	      $$\F \phi(k')=\begin{cases}
			      0 & \mbox{ for }  |k'|\geq 1\\
			      1 & \mbox{ for } |k'|\leq 1
			     \end{cases}$$					
	      and such that $\int_{\R^{d-1}} \phi(x')dx'=1.$

	      We claim that, under assumption (\ref{B1}), there exists $\psi \in L^1(\R^{d-1})$ such that 
	      \begin{equation}\label{O}
		(\rm{Id}-\phi\ast')r= \psi\ast'\nabla r\,.
	      \end{equation}
	      Since $r=r-\phi\ast r$, if we assume (\ref{O}) the conclusion follows from Young's inequality
	      $$ \int_{\R^{d-1}}|r(x',z)|dx'\leq \int_{\R^{d-1}}|\psi(x')|dx'\int_{\R^{d-1}}|\nabla r(x',z)|dx'\,.$$
	      \underline{Argument for (\ref{O}):}
	      \newline
	      Using the assumptions on $\phi$ and performing suitable change of variables, we find
	      \begin{eqnarray*}
	      &&r(x',z)-\int \phi(x'-y')r(y',z) dy'\\
	                    &=&\int\phi(x'-y')(r(x',z)-r(y',z)) dy'\\
                           &=&\int_{\R^{d-1}}\phi(x'-y')\int_0^1 (x'-y')\nabla' r(tx'+(t-1)(x'-y'), z) dy'dt\\
                           &=&\int_0^1\int_{\R^{d-1}}\phi(\xi)\nabla'r(x'+(t-1)\xi,z)\cdot \xi d\xi dt\\
                           &=&\int_0^1\int_{\R^{d-1}}\phi\left(\frac{\hat y'-x'}{t}\right)\nabla r(\hat y',z)\cdot \frac{\hat y'-x'}{t} dt \frac{1}{t^{d-1}}d\hat y'\\
                           &=&\int_{\R^{d-1}}\nabla' r(\hat y',z)\cdot \left(\int_0^1\phi\left(\frac{\hat y'-x'}{t}\right)\frac{\hat y'-x'}{t^{d}} dt\right) d\hat y'\\
                           &=&\int_{\R^{d-1}}\nabla' r(\hat y',z)\psi\left(\frac{\hat y'-x'}{t}\right)d\hat y',
	     \end{eqnarray*}
	     where 
	    $$\psi(x')=\int_0^1\phi\left(\frac{-x'}{t}\right)\frac{x'}{t^{d}}dt\,.$$
	    We notice that $\psi\in L^1(\R^{d-1})$, in fact
	    $$\int_{\R^{d-1}}|\psi(x')|dx'\leq \int_0^1 \int_{\R^{d-1}}|\phi(x'/t)\frac{x'}{t^{d}}| dx' dt=\int_{\R^{d-1}}|\phi(\xi)\xi|d\xi\,.$$
  
  \item[b)]
	  In Fourier space we have 
	  $$\F\nabla'r(k',z)=ik'\F r(k',z)=R^{-1}\F G(Rk')\F r(k',z)=R^{-1}\F G_R(k')\F r(k',z),$$
	  where $G$ is a Schwartz function and $G_{R}(x')=R^{-d}\F G(x'/R)$. Since 
	  $\int |G_R|dx'=\int |G| dx'$ is independent of $R$, we may conclude by Young
	  $$\int |\nabla' r| dx'\leq \frac{1}{R}\int |G_R| dx'\int |r| dx'\lesssim \frac{1}{R}\int |r| dx'\,. $$
	  
  \end{enumerate}

 \end{proof}

       
Here we prove an elementary estimate that will be applied in the argument for (\ref{A1.1}) and (\ref{A2.1}), Lemma \ref{lemma3}
\begin{lemma}\label{Lemma5}\ \\
Let $K=K(z)$ be a real function and define
 $$\overline{K}(z,\tilde z)=\frac{\tilde z}{z}|K(\tilde z- z)-K( z+\tilde z)|\,.$$
Then
\begin{equation}\label{EE1}
\sup_{\tilde z}\int_0^{\infty}\overline{K}(z,\tilde z) dz\lesssim \int_{\R}|K(z)|dz+\sup_{z\in \R}(z^2|\partial_zK(z)|)\,.
\end{equation}
 \end{lemma}
 
 \begin{proof}
Let us distinguish two regions:  $\frac{1}{2}\left|\frac{\tilde{z}}{z}\right|<1$ and $\frac{1}{2}\left|\frac{\tilde{z}}{z}\right|>1$. 
\newline
For $|z|\geq \frac{1}{2}|\tilde z|$ we have
\begin{eqnarray*}
&&\sup_{\tilde z}\int_{|z|\geq \frac{1}{2}|\tilde z|}|\overline{K}(z,\tilde z)|dz\\
&\leq &\max_{\tilde z}\int_{|z|\geq \frac{1}{2}|\tilde z|}|K(\tilde z-z)-K(z+\tilde z)|dz\lesssim \int|K(z)|dz\,.\\
\end{eqnarray*}
While for the region $|z|\leq \frac{1}{2}|\tilde z|$ we have,
\begin{eqnarray*}
&&\max_{\tilde z}|\tilde z|\int_{|z|\leq\frac{1}{2}|\tilde z|}\frac{1}{|z|}|K(\tilde z-z)-K(z+\tilde z)|dz\\
&=&\max_{\tilde z}|\tilde z|\int_{|z|\leq\frac{1}{2}|\tilde z|}\frac{1}{|z|}\left|\int_{-1}^1K'(\tilde z+t z) zdt\right|dz\\
&\leq&\max_{\tilde z}|\tilde z|\int_{-1}^1\frac{1}{t}\int_{|z|\leq\frac{t}{2}|\tilde z|}|K'(\tilde z+ z) |dzdt\\
&\stackrel{\frac{1}{2}|\tilde z|\leq |\tilde z+z|}{\leq}& \max_{\tilde z}\int_{-1}^1\frac{1}{t}\int_{|z|\leq\frac{t}{2}|\tilde z|}2|\tilde z+z||K'(\tilde z+ z)|dt dz\\
&\leq&\max_{\tilde z}\int_{-1}^1\frac{2}{t}\max_{|z|\leq\frac{t}{2}|\tilde z|}\left\{|\tilde z+z||K'(\tilde z+ z)|\right\}\left(\int_{|z|\leq\frac{t}{2}|\tilde z|} dz\right)dt\\
&=&\max_{\tilde z}\int_{-1}^1\frac{1}{t}\max_{|z|\leq\frac{t}{2}|\tilde z|}\{|\tilde z+z||K'(\tilde z+ z)|\} t|\tilde z|dt\\
&=&2\max_{\tilde z}|\tilde z|\max_{|z|\leq\frac{t}{2}|\tilde z|}\{|\tilde z+z||K'(\tilde z+ z)|\}\\
&\stackrel{\frac{1}{2}|\tilde z|\leq |\tilde z+z|}{\leq}&4\max_{\tilde z}\max_{|z|\leq\frac{t}{2}|\tilde z|}\{| z+\tilde z|^2|K'(\tilde z+ z)|\}\,.\\ 
\end{eqnarray*}
In conclusion we have 
\[\max_{ z}\int|\bar K(z,\tilde z)|dz\lesssim\int|K(z)|dz+\max_{ z}|z|^2|K'(z)| \,.\]
\end{proof}

\subsection{Heat kernel: elementary estimates}
In this section we recall the definition of the heat kernel and some 
properties and estimates that we will use throughout the paper.
\newline
The function $\Gamma:\R^{d}\times \R\rightarrow \R$ is defined as
$$\Gamma(x,t)=\frac{1}{t^{d/2}}\exp(-\frac{|x|^2}{4t})$$
and we can rewrite it as 
$$\Gamma(x,t)=\Gamma_1(z,t)\Gamma_{d-1}(x',t) \qquad x'\in \R^{d-1}, z\in \R,$$
 where $$\Gamma_1(z,t)=\frac{1}{t^{1/2}}\exp(-\frac{z^2}{4t})$$ and 
 $$\Gamma_{d-1}(z,t)=\frac{1}{t^{(d-1)/2}}\exp(-\frac{|x'|^2}{4t})\,.$$ 
 
Here we list the bounds on the derivatives of $\Gamma$ that are used in Section \ref{App}, Lemma \ref{lemma3}: 
\begin{enumerate}
  \item  \begin{equation}\label{z0}
         \langle|(\nabla')^n\Gamma_{d-1}|\rangle'\approx \frac{1}{t^{\frac{n}{2}}}\,.
         \end{equation}
 \item  \begin{equation}\label{x1}
        \int_{\R}|\partial_z^{n}\Gamma_1|dz\lesssim \frac{1}{t^{\frac{n}{2}}}\,.
        \end{equation}
\item   \begin{equation}\label{y1}
        \int_{0}^{\infty}|\partial_z\Gamma_{1}(z,t)|dt=\int_0^{\infty}\left|\frac{1}{\hat{t}^{3/2}}\exp{(-\frac{1}{4\hat{t}})}\right|d\tilde t\lesssim 1\,,
        \end{equation}
        where we have used the change of variable $\hat{t}=\frac{t}{z^2}$.
\item   \begin{equation}\label{y2}
        \sup_{z\in \R}\left(z|\partial_z\Gamma_1(z,t)|\right)=\sup_{\xi}\left|\frac{1}{t^{\frac{1}{2}}}\xi^{2}\exp^{-\xi^2}\right|\lesssim \frac{1}{t^{\frac{1}{2}}}\,,
        \end{equation}
        where we have used the change of variable $\xi=\frac{z}{t^{\frac{1}{2}}}$\,.
\item   \begin{equation}\label{y3}
        \sup_{z\in \R}\left(z^2|\partial_z\Gamma_1(z,t)|\right)=\sup_{\xi}\left|\xi^{3}\exp^{-\xi^2}\right|\lesssim 1\,,
        \end{equation}
       where we have used the change of variable $\xi=\frac{z}{t^{\frac{1}{2}}}$\,.
\end{enumerate}



\bibliographystyle{unsrt}
\bibliography{Biblio}

\end{document}